\documentclass[12pt]{amsart}
\usepackage{amsmath, amssymb, amsfonts}
\usepackage[all]{xypic}
\usepackage[pdftex]{graphicx}

\theoremstyle{plain}
\newtheorem{theorem}{Theorem}[section]
\newtheorem{lemma}[theorem]{Lemma}
\newtheorem{corollary}[theorem]{Corollary}
\newtheorem{prop}[theorem]{Proposition}

\theoremstyle{remark}

\newtheorem{remark}[theorem]{Remark}

\newtheorem{example}[theorem]{Example}

\newtheorem*{note*}{Note}
\newtheorem*{remark*}{Remark}
\newtheorem*{example*}{Example}

\theoremstyle{definition}
\newtheorem*{definition*}{Definition}
\newtheorem{definition}[theorem]{Definition}

\newcommand{\Z}{\mathbb{Z}}

\newcommand{\Q}{\mathbb{Q}}

\newcommand{\N}{\mathbb{N}}

\newcommand{\Ann}{\mathrm{Ann}}

\newcommand{\Gal}{\mathrm{Gal}}
\newcommand{\Tr}{\mathrm{Tr}}
\newcommand{\Gl}{\mathrm{GL}}

\newcommand{\cl}{\mathrm{cl}}

\newcommand{\im}{\mathrm{im}}

\newcommand{\coker}{\mathrm{coker}}
\newcommand{\nr}{\mathrm{nr}}

\newcommand{\Hom}{\mathrm{Hom}}

\newcommand{\Fit}{\mathrm{Fit}}
\newcommand{\tr}{\mathrm{tr}}

\numberwithin{equation}{section}

 \newcommand{\onto}{\twoheadrightarrow}
 \newcommand{\Fitt}{\mathrm{Fitt}}
 \newcommand{\Br}{\mathrm{Br}}

\title[Noncommutative Fitting invariants and improved annihilation results]{Noncommutative Fitting invariants\\and improved annihilation results}

\author{Henri Johnston}
\address{Henri Johnston\\
St.\ John's College\\
Cambridge CB2 1TP\\
United Kingdom
}
\email{H.Johnston@dpmms.cam.ac.uk}
\urladdr{http://www.dpmms.cam.ac.uk/$\sim$hlj31}

\author{Andreas Nickel}
\address{Andreas Nickel\\
Universit\"{a}t Bielefeld\\
Fakult\"{a}t f\"{u}r Mathematik\\
Postfach 100131\\
Universit\"{a}tsstr. 25\\
33501 Bielefeld\\
Germany}
\email{anickel3@math.uni-bielefeld.de}
\urladdr{http://www.math.uni-bielefeld.de/$\sim$anickel3/english.html}

\date{Version of 23rd January 2013}
\thanks{The second author acknowledges financial support provided by the DFG.}
\begin{document}

\maketitle

\begin{abstract}
To each finitely presented module $M$ over a commutative ring $R$
one can associate an $R$-ideal $\Fit_R(M)$ which is called the (zeroth) Fitting ideal of $M$ over $R$ and which is always contained in the $R$-annihilator of $M$.
In an earlier article, the second named author generalised this notion by
replacing $R$ with a (not necessarily commutative) $\mathfrak{o}$-order $\Lambda$ in a finite dimensional separable algebra, where $\mathfrak{o}$ is an integrally closed complete commutative noetherian local domain.
To obtain annihilators, one has to multiply
the Fitting invariant of a (left) $\Lambda$-module $M$ by a certain ideal $\mathcal{H}(\Lambda)$
of the centre of $\Lambda$.
In contrast to the commutative case, this ideal can be properly contained in the centre of $\Lambda$.
In the present article, we determine explicit lower bounds for $\mathcal{H}(\Lambda)$ in many cases.
Furthermore, we define a class of `nice' orders $\Lambda$ over which Fitting invariants
have several useful properties such as good behaviour with respect to direct sums of modules,
computability in a certain sense, and $\mathcal{H}(\Lambda)$ being the best possible.
\end{abstract}

\section{Introduction}

Let $R$ be a commutative ring (with identity) and let $M$ be a finitely presented $R$-module.
If we choose a presentation
\begin{equation}\label{eqn:intro-presentation}
R^{a} \stackrel{h}{\longrightarrow} R^{b} \onto M
\end{equation}
we may identify the homomorphism $h$ with an $a \times b$ matrix with entries in $R$.
If $a \geq b$, the (zeroth) Fitting ideal of $M$ over $R$, denoted by $\Fit_{R}(M)$, is
defined to be the $R$-ideal generated by all $b \times b$ minors of the matrix corresponding to
$h$.
If $a<b$ then $\Fit_{R}(M)$ is defined to be the zero ideal of $R$.
A key point is that this definition is independent of the choice of presentation $h$.
This notion was introduced by H.\ Fitting \cite{0016.05003} and is now a very important
tool in commutative algebra thanks to several useful properties.
In particular, $\Fit_{R}(M)$ is always a subset of $\Ann_{R}(M)$, the $R$-annihilator of $M$.
Furthermore, $\Fit_{R}(M)$ is often computable, thanks to being independent of the choice of
presentation $h$ and, for example, good behaviour with respect to quotients of $R$, as well as
epimorphisms and direct sums of $R$-modules.
For a full account of the theory, we refer the reader to \cite{MR0460383}.

It is natural to ask whether analogous invariants can be defined for modules over noncommutative rings;
indeed, there have been several attempts to overcome the technical obstacles involved in order to do this.
In \cite{MR959761} and \cite{MR1013481}, J.\ Susperregui considered two particular cases:
skewcommutative graded rings and rings of differential operators satisfying the left Ore property.
In his Ph.D.\ thesis  \cite{grime_thesis}, P.\ Grime considered several cases including matrix rings
over commutative rings, as well as certain hereditary orders and (twisted) group rings.
We say that a (left) $R$-module $M$ has a quadratic presentation if one can take $a=b$ in
\eqref{eqn:intro-presentation}.
In the case where  $G$ is a finite group and $R$ is a group ring $\Z[G]$, $\Z_{(p)}[G]$, or $\Z_{p}[G]$
for some prime $p$,  A.\ Parker in his  Ph.D.\ thesis \cite{parker_thesis} defined noncommutative Fitting
invariants for modules with a quadratic presentation.

Let $A$ be a finite dimensional separable algebra over a field $F$ and $\Lambda$ an $\mathfrak{o}$-order in $A$,
where $\mathfrak{o}$ is an integrally closed complete commutative noetherian local domain with
field of quotients $F$. We call such an order $\Lambda$ a Fitting order; a standard example is the group ring $\Z_{p}[G]$ where
$p$ is a prime and $G$ is a finite group. We denote by $\zeta(A)$ and $\zeta(\Lambda)$ the centres
of $A$ and $\Lambda$, respectively.
All modules are henceforth assumed to be left modules unless otherwise stated.
Let $M$ be a $\Lambda$-module admitting a finite
presentation
\[
\Lambda^{a} \stackrel{h}{\longrightarrow} \Lambda^{b} \onto M.
\]
In \cite{MR2609173}, the Fitting invariant $\Fitt_{\Lambda}(h)$ is defined to be
an equivalence class of a certain $\zeta(\Lambda)$-submodule of $\zeta(A)$ generated by reduced norms.
In the case that $\Lambda$ is commutative, the reduced norm is the same as the usual determinant
and this notion is compatible with the classical definition of Fitting ideal described above.
In contrast to the commutative case, $\Fitt_{\Lambda}(h)$ does in general depend on $h$;
however, for a given $M$ there exists a distinguished Fitting invariant $\Fitt_{\Lambda}^{\max}(M)$
that is maximal among all $\Fitt_{\Lambda}(h)$.
Moreover, if $M$ admits a quadratic presentation $h$, then $\Fitt_{\Lambda}(h)$ is independent of the choice of
$h$ (as long as $h$ is quadratic) and the definition is compatible with that given by A.\ Parker in
his thesis \cite{parker_thesis}. It is also shown in \cite{MR2609173} that $\Fitt_{\Lambda}^{\max}(M)$
enjoys many of the useful properties of the commutative case (see Theorem \ref{thm:fitt-thm}).
To obtain annihilators from $\Fitt_{\Lambda}^{\max}(M)$, one has to multiply by a certain ideal $\mathcal{H}(\Lambda)$
of $\zeta(\Lambda)$; if $\Lambda$ is commutative or maximal, then $\mathcal{H}(\Lambda)=\zeta(\Lambda)$,
but in general $\mathcal{H}(\Lambda)$ is a proper ideal of $\zeta(\Lambda)$.
Though much progress is made in \cite{MR2609173}, several questions remain:
\begin{enumerate}
\item Can $\mathcal{H}(\Lambda)$ be computed or approximated explicitly?
\item Does $\Fitt_{\Lambda}^{\max}(M)$ behave well with respect to direct sums of $\Lambda$-modules?
\item For a left ideal $I$ of $\Lambda$, can we give an explicit formula for $\Fitt_{\Lambda}^{\max}(\Lambda/I)$?
\item Are there certain Fitting orders $\Lambda$ for which $\Fitt^{\max}_{\Lambda}(M)$ can be computed from a presentation $h$ of $M$, independently of the choice of $h$?
\end{enumerate}
The present article goes some way towards answering these questions.

The motivation behind the theory of noncommutative Fitting invariants comes from arithmetic. 
In \cite{MR2771125} special values of $L$-functions attached to (not necessarily abelian) 
Galois extensions of number fields were used to construct annihilators of ideal class groups. 
In \cite{MR2609173} noncommutative Fitting invariants were used to predict similar annihilators 
under the assumption of the relevant special case of the equivariant Tamagawa number conjecture (ETNC)
of Burns and Flach (\cite{MR1884523}, \cite{MR1863302}). 
The results of the present article can be used to make these annihilation results more explicit 
(see Remark \ref{rmk:non-ab-stick-improvement} for a more detailed account and Remark \ref{rmk:arith-more-explicit}
for further examples of this kind). 
More generally, noncommutative Fitting invariants appear to be the natural formalism which arises when 
one attempts to derive concrete consequences of the abstract formalism of either the ETNC or the 
main conjectures of noncommutative Iwasawa theory 
(see \cite{EIMC}, \cite{MR2801311}, \cite{nonabstark},  \cite{MR2822866}).

We now describe the contents and main results in more detail.
In \S \ref{sec:matrix-over-commutative} we consider the case of a matrix ring $\Lambda$ over an arbitrary commutative ring $R$ (with identity).
We use explicit Morita equivalence of $\Lambda$ and $R$ to define an ideal of $R$
(the definition is essentially equivalent to that of \cite[\S 5.2]{grime_thesis}),
and go on to establish a number of useful properties.
This ideal is equal to the usual Fitting ideal in the commutative case (i.e. $\Lambda=R$).
We also give a slight sharpening of an existing result on classical Fitting ideals.
In \S \ref{sec:non-comm-fitt} we review background material and the main results of \cite{MR2609173}.
We return to the situation in which $\Lambda$ is a Fitting order contained in $A$ and
introduce $\Fit_{\Lambda}(h)$ as an alternative to $\Fitt_{\Lambda}(h)$.
The former is a $\zeta(\Lambda)$-submodule of $\zeta(A)$ whereas the latter (originally introduced
in \cite{MR2609173}) is an equivalence class of such modules; the two definitions are closely related.
We define $\Fit_{\Lambda}^{\max}(M)$ analogously to $\Fitt_{\Lambda}^{\max}(M)$.
Furthermore, we show that $\Fit_{\Lambda}^{\max}(M)$ is equal to the ideal defined in \S \ref{sec:matrix-over-commutative} when $\Lambda$
is both a Fitting order and a matrix ring over a commutative ring.
In \S \ref{sec:nice-fItting-orders} we introduce the notion of a `nice' Fitting order.
A Fitting order is defined to be nice if it is a finite direct sum of maximal orders and matrix rings over commutative rings.
Such an order has particularly useful properties; indeed,
the answer to each of questions (i)-(iv) above is affirmative in this case.
In particular, if $\Lambda$ is nice then
$\mathcal{H}(\Lambda)=\zeta(\Lambda)$ and so
$\Fit_{\Lambda}^{\max}(M)$ is always a subset of $\Ann_{\zeta(\Lambda)}(M)$.
We show that if $p$ is a prime and $G$ is a finite group then the group ring $\Z_{p}[G]$
is a nice Fitting order if and only if $p$ does not divide the order of the commutator subgroup $G'$.
Moreover, we show a similar result for completed group algebras $\Z_{p} [[ G ]]$, where $G$ is
a $p$-adic Lie group of dimension $1$.
In \S \ref{sec:quots-by-left-ideals} we explicitly compute the maximal Fitting invariant of the quotient of a Fitting order $\Lambda$ by a left ideal $I$ when either $\Lambda$ is nice or $I$ is principal;
we give a containment in other cases.
In \S \ref{sec:ann-and-change-of-order} we compute certain conductors and thereby give explicit bounds for $\mathcal{H}(\Lambda)$ in the case that $\Lambda$ is not nice;
we also give further annihilation results relating to change of order.

\subsection*{Notation and conventions}
All rings are assumed to have an identity element and all modules are assumed
to be left modules unless otherwise  stated. We denote the set of all $m \times n$
matrices with entries in a ring $R$ by $M_{m \times n} (R)$ and in the case $m=n$
the group of all invertible elements of $M_{n \times n} (R)$ by $\Gl_{n}(R)$.
We write $\zeta(R)$ for the centre of $R$ and $K_{1}(R)$ for the Whitehead group
(see \cite[\S 40]{MR892316}).

\subsection*{Acknowledgements}
The authors are grateful to Cornelius Greither for several useful comments and suggestions, to
Steve Wilson for providing a copy of Peter Grime's Ph.D.\ thesis \cite{grime_thesis}, 
and to the referee for several corrections and useful suggestions regarding the exposition.

\section{Matrix rings over commutative rings}\label{sec:matrix-over-commutative}

Let $R$ be a commutative ring and fix $n \in \N$.
Let $\Lambda = M_{n \times n}(R)$ and for $1 \leq i,j \leq n$ let $e_{ij} \in \Lambda$
be the matrix with $1$ in position $(i,j)$ and $0$ everywhere else.
Then
\[
e_{ij}e_{kl} = \left\{ \begin{array}{ll}
e_{il} & \textrm{if } j=k,      \\
0 & \textrm{otherwise}.
\end{array}
\right.
\]

\begin{definition}\label{def:fitt-matrix}
Let $M$ be a finitely presented $\Lambda$-module.
Then define
\[
\Fit_{\Lambda}(M) := \Fit_{R}(e_{11}M),
\]
where the right hand side denotes the usual Fitting ideal over a commutative ring.
\end{definition}

\begin{remark}
In the case $n=1$ we have $\Lambda = R$ and $e_{11}=1$, so Definition \ref{def:fitt-matrix}
is just the standard definition in this case and hence our notation is consistent.
\end{remark}

\begin{lemma}\label{lem:e_ii-isoms}
Let $M$ be a $\Lambda$-module.
For $1 \leq i,j \leq n$ we have $e_{ii}M \simeq e_{jj}M$ as $R$-modules.
\end{lemma}

\begin{proof}
Define an $R$-module homomorphism  $\alpha_{ij}: e_{ii}M \rightarrow e_{jj}M$ by $x \mapsto e_{ji} x$.
Note that this is in fact well-defined since $e_{ji}M = e_{jj}e_{ji}M \subset e_{jj}M$. Define $\alpha_{ji}$ symmetrically.
Then
\[
\alpha_{ji} \circ \alpha_{ij} (x) = e_{ij}e_{ji} x = e_{ii} x = x.
\]
So by symmetry $\alpha_{ij}$ and $\alpha_{ji}$ are mutually inverse and hence are isomorphisms.
\end{proof}

We give some of the important properties of Fitting ideals over $\Lambda$.

\begin{theorem}\label{thm:properties-matrix-over-comm}
Let $M$, $M_{1}$, $M_{2}$ and $M_{3}$ be finitely presented $\Lambda$-modules.
\begin{enumerate}
  \item For any $1 \leq i \leq n$, we have $\Fit_{\Lambda}(M) = \Fit_{R}(e_{ii}M)$.
  \item We have $\Fit_{\Lambda}(M) \subset \Ann_{R}(M)$.
  \item If $M_{1} \onto M_{2}$ is an epimorphism then
  $\Fit_{\Lambda}(M_{1}) \subset \Fit_{\Lambda}(M_{2})$.
  \item If $M_{2} = M_{1} \oplus M_{3}$ then
  $ \Fit_{\Lambda}(M_{2}) = \Fit_{\Lambda}(M_{1}) \cdot \Fit_{\Lambda}(M_{3})$.
  \item If $M_{1} \stackrel{\iota}{\rightarrow} M_{2} \onto M_{3}$
  is an exact sequence ($\iota$ need not be injective)
  then \[ \Fit_{\Lambda}(M_{1}) \cdot \Fit_{\Lambda}(M_{3}) \subset \Fit_{\Lambda}(M_{2}). \]
  \item If $M_{1} \hookrightarrow M_{2} \onto M_{3}$ is an exact sequence and
   $M_{3}$ has a quadratic presentation
   (i.e. of the form $\Lambda^{k} \rightarrow \Lambda^{k} \onto M_{3}$ for some $k \in \N$) then
  \[ \Fit_{\Lambda}(M_{1}) \cdot \Fit_{\Lambda}(M_{3}) = \Fit_{\Lambda}(M_{2}). \]
  \item For any map $R \rightarrow S$ of commutative rings we have
  \[ \Fit_{S \otimes_{R} \Lambda}(S \otimes_{R} M) = S \cdot \Fit_{\Lambda}(M). \]
  \item We have $\Fit_{R}(M) = \Fit_{\Lambda}(M)^{n}$.
  \item If $I$ is a finitely generated two-sided ideal of $\Lambda$ then $I=M_{n \times n}(J)$
  for some ideal $J$ of $R$ and so $\Lambda/I = M_{n \times n}(R/J)$;
  hence we have $\Fit_{\Lambda}(\Lambda/I) = J^{n}$.
\end{enumerate}
\end{theorem}

\begin{remark}\label{rmk:fit-cent-fit-lambda}
If $R$ is a Dedekind domain then factorisation of ideals in $R$ is unique and so
Theorem \ref{thm:properties-matrix-over-comm} (viii) shows that $\Fit_{\Lambda}(M)$
is completely determined by $\Fit_{R}(M)$ in this case.
\end{remark}

\begin{remark}
We note that $\Ann_{\Lambda}(M) := \{ x \in \Lambda \mid x \cdot M=0 \}$ is always a two-sided ideal of $\Lambda$ and from this it is straightforward to show
that $\Ann_{\Lambda}(M) = M_{n \times n}(\Ann_{R}(M))$.
Thus nothing is lost by computing or approximating $\Ann_{R}(M)$ rather than $\Ann_{\Lambda}(M)$.
\end{remark}

\begin{proof}
Definition \ref{def:fitt-matrix} and Lemma \ref{lem:e_ii-isoms} give (i).
For (ii), note that $e_{11}+\cdots+e_{nn}$ is the identity matrix in $\Lambda$
and that $e_{ii}M \cap e_{jj}M = 0$ for $i \neq j$.
Hence as $R$-modules
\begin{equation}\label{eqn:break-up-M}
M = (e_{11}+\cdots+e_{nn})M = e_{11}M  \oplus \cdots \oplus e_{nn} M.
\end{equation}
By (i) and the annihilation property of Fitting ideals over $R$, we have
$\Fit_{\Lambda}(M) = \Fit_{R}(e_{ii}M) \subset \Ann_{R}(e_{ii}M)$ for each $i$
and therefore $\Fit_{\Lambda}(M) \subset \Ann_{R}(M)$.

Equation \eqref{eqn:break-up-M} shows that $M \mapsto e_{11}M$ is an exact covariant functor from
the category of (left) $\Lambda$-modules to $R$-modules. (Note that this functor takes a
$\Lambda$-homomorphism $M \rightarrow N$ to its restriction $e_{11}M \rightarrow e_{11}N$
considered as an $R$-homomorphism.)
Furthermore, $e_{11}\Lambda \simeq R^{n}$ as $R$-modules,
so free (resp.\ finitely presented) $\Lambda$-modules map to free (resp.\ finitely presented) $R$-modules.
Therefore (iii)-(vii) follow from the corresponding properties for Fitting ideals over $R$.
Proofs of (iii) and (iv) in the case $\Lambda=R$ can be found in \cite[Chapter 3]{MR0460383};
for (vii) see  \cite[Corollary 20.5]{MR1322960}. Properties (v) and (vi) follow from
Lemma \ref{lem:short-exact-fitting} below.
Note that for (v), we first reduce to the case that $\iota$ is injective:
as $M_{1}$ surjects onto $\ker(M_{2} \onto M_{3})$ by exactness, we can assume by (iii)
that in fact $M_{1}=\ker(M_{2} \onto M_{3})$.
Property (viii) follows from equation \eqref{eqn:break-up-M},
Lemma \ref{lem:e_ii-isoms}, and (iv) in the case $\Lambda=R$.
The first part of (ix) is well-known; the second part now follows from
the $R$-module isomorphism $e_{11}(\Lambda/I) \simeq (R/J)^{n}$,
the fact that $\Fit_{R}(R/J)=J$ (see \cite[\S 3.1, Exercise 4]{MR0460383}; solution on p.93), and  parts (i) and (iv).
\end{proof}

\begin{example}
Let $n=2$ and $R=\Z$ so that $\Lambda = M_{2 \times 2}(\Z)$.
Consider $M = M_{2 \times 2}(\Z/2\Z)$ as a $\Lambda$-module.
Then $\Fit_{\Z}(M) = 16\Z$, $\Fit_{\Lambda}(M)=4\Z$, and $\Ann_{\Z}(M)=2\Z$.
Now let $N = Me_{11}$.
Then $\Fit_{\Z}(N) = 4\Z$ and $\Fit_{\Lambda}(N)=\Ann_{\Z}(N)=2\Z$.
\end{example}

\begin{remark}\label{rmk:morita-equiv}
The key fact we have used is that $R$ and $\Lambda$ are Morita equivalent rings
(for background on Morita equivalence see \cite[\S 3D]{MR632548}, \cite[Chapter 4]{MR1972204} or \cite[Chapter 7]{MR1653294}).
Let ${}_{R}\mathfrak{M}$ and ${}_{\Lambda}\mathfrak{M}$ denote the
categories of (left) $R$ modules and left $\Lambda$-modules, respectively. Fix $1 \leq i \leq n$.
Then we have mutually inverse category equivalences
\[
F : {}_{\Lambda}\mathfrak{M} \longrightarrow {}_{R}\mathfrak{M} \quad \textrm{ and } \quad
G : {}_{R}\mathfrak{M} \longrightarrow {}_{\Lambda}\mathfrak{M}
\]
given explicitly by
\begin{eqnarray}
\label{eqn:morita-isos}
F(M) &=&  e_{ii}\Lambda \otimes_{\Lambda} M \simeq  e_{ii}M \simeq \Hom_{\Lambda}(\Lambda e_{ii},M),\\
G(N) &=& \Lambda e_{ii} \otimes_{R} N \simeq \Hom_R (e_{ii} \Lambda, N). \notag
\end{eqnarray}
The $R$-module isomorphisms of \eqref{eqn:morita-isos} can be used to give definitions equivalent to
Definition \ref{def:fitt-matrix}.
In fact, in his PhD thesis \cite[\S 5.2]{grime_thesis}, Peter Grime essentially defines the Fitting ideal
of a $\Lambda$-module $M$ to be $\Fit_{R}(\Hom_{\Lambda}(\Lambda e_{11},M))$. However,
most of his results are quite different to those given here.
\end{remark}

\begin{remark}
We note that it is straightforward to extend Definition \ref{def:fitt-matrix} and parts (ii)-(ix) of 
Theorem \ref{thm:properties-matrix-over-comm}
to the case where $\Lambda$ is any ring that is Morita equivalent to a commutative ring $R$.
The advantages of the more specific case described in this section are that
it is very explicit, and thus is easier to understand and more results can be obtained.
Note that if $R$ is a ring over which every finitely generated projective module is in fact free
(for example, a principal ideal domain or a local ring) then we must have
$\Lambda  \simeq M_{n \times n}(R)$ for some $n$,
and so this case is covered by Definition \ref{def:fitt-matrix}.
In fact, from \S \ref{sec:non-comm-fitt} onwards we shall work over a ring $\Lambda$ 
whose centre $\zeta(\Lambda)$ is a product of local rings; we can
without loss of generality suppose that $\zeta(\Lambda)$ is in fact local.
Since $\Lambda$ is Morita equivalent to $R$, we have $\zeta(\Lambda) \simeq \zeta(R)=R$;
therefore $\Lambda \simeq M_{n \times n}(R)$ for some $n$.
Thus the more general case is not needed for this article.
\end{remark}

The following technical lemma is essentially equivalent to
\cite[Lemma 5.1]{grime_thesis}.

\begin{lemma}\label{lem:lambda-matrix-to-R}
Fix $1 \leq i \leq n$ and note that $\mathcal{B}_{i} := \{ e_{ij} \}_{1 \leq j \leq n}$ is an $R$-basis of $e_{ii}\Lambda$.
For any $r,s \in \N$ and any $\Lambda$-homomorphism
$\alpha : \Lambda^{r} \longrightarrow \Lambda^{s}$,
let $\alpha': (e_{ii}\Lambda)^{r} \longrightarrow (e_{ii}\Lambda)^{s}$
be the restriction of $\alpha$ considered as an $R$-homomorphism.
Let $h : \Lambda^{a} \longrightarrow \Lambda^{b}$ be a $\Lambda$-homomorphism
represented by $H \in M_{a \times b}(\Lambda)$ with respect to the standard basis.
Let $H' \in M_{na \times nb}(R)$ be the matrix representing $h'$ with respect to
the bases of $(e_{ii}\Lambda)^{a}$ and $(e_{ii}\Lambda)^{b}$ obtained
from $\mathcal{B}_{i}$ in the obvious way.
Let $\tilde{H} \in M_{na \times nb}(R)$ be the same matrix as $H$ but with entries considered
in $R$ rather than $\Lambda$. Then $H'=\tilde{H}$.
\end{lemma}

\begin{proof}
Fix $1 \leq k \leq a$ and $1 \leq \ell \leq b$.
Let $\iota_{k}: \Lambda \longrightarrow \Lambda^{a}$ be the obvious injection
and $\pi_{\ell}: \Lambda^{b} \longrightarrow \Lambda$ be the obvious projection.
Then  $\iota_{k}'$ (resp. $\pi_{\ell}'$) is also the obvious injection (resp. projection).
Let $h_{k \ell} = \pi_{\ell} \circ h \circ \iota_{k} : \Lambda \longrightarrow \Lambda$.
Then $h_{k \ell}' = \pi_{\ell}' \circ h' \circ \iota_{k}'$.
Hence we can and do assume without loss of generality that $a=b=1$.

Write $\tilde{H} = (r_{pq}) \in M_{n \times n}(R)=\Lambda$.
Then for $1 \leq j \leq n$ we have
\[
h'(e_{ij}) = e_{ij} H = e_{ij} \sum_{p,q=1}^{n} e_{pq} r_{pq}
= \sum_{p,q=1}^{n} e_{ij}e_{pq} r_{pq} = \sum_{q=1}^{n} e_{iq} r_{jq}.
\]
Hence $H'$ is the matrix $(r_{jq})_{j,q} = \tilde{H}$, as required.
\end{proof}

\begin{remark}
Lemma \ref{lem:lambda-matrix-to-R} can be used to give an alternative proof
of Theorem \ref{thm:properties-matrix-over-comm} (i).
\end{remark}

\begin{prop}\label{prop:matrix-ring-ideal-det}
Let $I$ be a finitely generated left ideal of $\Lambda$.
Then
\[
\Fit_{\Lambda}(\Lambda/I) = \langle \det(x) \mid x \in I \rangle_{R}.
\]
\end{prop}

\begin{proof}
We adopt the notation and assume the result of Lemma \ref{lem:lambda-matrix-to-R}.
Let $\{ x_{1}, \ldots, x_{r-1} \}$ be a fixed set of generators of $I$ and let $x_{r}$ be an arbitrary element of $I$.
Then there exists a presentation of $\Lambda/I$ of the form
\[
\Lambda^{r} \stackrel{h}{\longrightarrow} \Lambda \onto \Lambda/I,
\]
where $H := (x_{1}, \ldots, x_{r})^{t}  \in M_{r \times 1}(\Lambda)$ is the matrix representing $h$.
Let $S$ denote the set of all $n \times n$ submatrices of $H'=\tilde{H}\in M_{nr \times n}(R)$.
Since $h'$ is an $R$-module presentation of $e_{11}(\Lambda/I)$
and $\Fit_{R}(e_{11}(\Lambda/I))$ is independent of the choice of presentation, we have
\[
\Fit_{\Lambda}(\Lambda/I) = \Fit_{R}(e_{11}(\Lambda/I)) = \langle \det(T) \mid T \in S \rangle.
\]
However, one of the elements of $S$ is equal to $x_{r}$,
and so we see that $\det(x_{r}) \in \Fit_{\Lambda}(\Lambda/I)$.
We therefore have $\langle \det(x) \mid x \in I \rangle_{R} \subset \Fit_{\Lambda}(\Lambda/I)$.

Now let $T \in S$. Fix $i$ with $1 \leq i \leq n$.
Then the $i$th row of $T$ is a row of $H'=\tilde{H}$, which in turn is the $j$th row
of $x_{k}$ for some $k,j$ with $1 \leq k \leq r$ and $1 \leq j \leq n$.
Hence $e_{ii}T = e_{ij}x_{k}$.
Since $x_{k} \in I$, $e_{ij} \in \Lambda$, and $I$ is a left ideal of $\Lambda$,
we thus have that $e_{ii}T \in I$.
Therefore $T=(e_{11}+\cdots+e_{nn})T = e_{11}T + \cdots + e_{nn}T \in I$,
and so $\Fit_{\Lambda}(\Lambda/I) \subset \langle \det(x) \mid x \in I \rangle_{R}$.
\end{proof}

\subsection{Auxiliary result on Fitting ideals over commutative rings}
Let $R$ be a commutative ring.
We provide a proof of the following result as the second part is slightly stronger than similar results
that the authors were able to locate in the literature.

\begin{lemma}\label{lem:short-exact-fitting}
Let $M_{1},M_{2}$ and $M_{3}$ be finitely presented $R$-modules.
\begin{enumerate}
\item If $M_{1} \stackrel{\iota}\hookrightarrow M_{2} \onto M_{3}$ is an exact sequence
then \[ \Fit_{R}(M_{1}) \cdot \Fit_{R}(M_{3}) \subset \Fit_{R}(M_{2}). \]
\item  If in addition $M_{3}$ has a quadratic presentation (i.e. of the form $R^{k} \rightarrow R^{k} \onto M_{3}$
for some $k \in \N)$ then in fact
\[
\Fit_{R}(M_{1}) \cdot \Fit_{R}(M_{3}) = \Fit_{R}(M_{2}).
\]
\end{enumerate}
\end{lemma}

\begin{remark}
Lemma \ref{lem:short-exact-fitting} (i) is well-known
(see \cite[Exercise 2, Chapter 3]{MR0460383}; solution on p.90-91).
Proofs of slightly weaker versions of Lemma \ref{lem:short-exact-fitting} (ii) can be found in
\cite[p.80-81]{MR0460383} or \cite[Lemma 3]{MR1658000}); these assume that
$M_{3}$ has a presentation of the form $R^{k} \stackrel{h}{\rightarrow} R^{k} \onto M_{3}$
with $h$ injective, whereas
Lemma \ref{lem:short-exact-fitting} (ii) does not  require $h$ to be injective.
\end{remark}

\begin{proof}
We choose presentations $R^{a_{i}} \stackrel{h_{i}}\longrightarrow R^{b_{i}} \stackrel{\pi_{i}}{\onto} M_{i}$ for $i=1,3$
and construct a finite presentation of $M_{2}$ in the following way. Since $R^{b_{3}}$ is projective, $\pi_{3}$ factors through $M_{2}$ via a map $f_{1}:R^{b_{3}} \rightarrow M_{2}$.
We define $\pi_{2}=(\iota \circ \pi_{1} \mid f_{1}) : R^{b_{1}} \oplus R^{b_{2}} \onto M_{2}$.
In a similar manner we construct
$h_{2}=	(h_{1} \mid f_{2}):R^{a_{1}} \oplus R^{a_{3}} \rightarrow R^{b_{1}} \oplus R^{b_{3}}$, where $f_{2}$ realises the factorisation of $h_{3}$ through $\ker(\pi_{2})$.
Let $a_{2}=a_{1}+a_{3}$ and $b_{2}=b_{1}+b_{3}$.
We identify each $h_{i}$ with multiplication on the right  by a matrix in $M_{a_{i} \times b_{i}}(R)$ in the obvious way.
Then $h_{2}$ is of the form
\[
\left(\begin{array}{cc} h_{1} & 0 \\ \ast & h_{3} \end{array}\right).
\]
Since Fitting ideals over $R$ are independent of the chosen presentation, this gives the desired inclusion of part (i).

Now suppose that $M_{3}$ has a quadratic presentation; then we can choose $a_{3}=b_{3}$.
Without loss of generality we can assume that $a_{1} \geq b_{1}$ and so $a_{2} \geq b_{2}$.
Let $H_{2}$ be a $b_{2} \times b_{2}$ submatrix of $h_{2}$. Then $H_{2}$ is obtained from
$h_{2}$ by deleting rows. If none of the last $a_{3}$ rows are deleted, then $H_{2}$
is of the form
\[
\left(\begin{array}{cc} H_{1} & 0 \\ \ast & h_{3} \end{array}\right),
\]
where $H_{1}$ is some $b_{1} \times b_{1}$ submatrix of $h_{1}$.
Otherwise, $H_{2}$ is of the form
\[
\left(\begin{array}{cc} A & 0 \\ \ast & B \end{array}\right),
\]
where $A$ and $B$ are square matrices ($B$ is a submatrix of $h_{3}$)
and the last column of $A$ consists only of zeros; hence $\det(H_{2})=\det(A)\det(B)=0$.
In either case, we have the reverse of the inclusion of part (i) and thus have the desired equality of part (ii).
\end{proof}

\section{Noncommutative Fitting invariants}\label{sec:non-comm-fitt}

\subsection{Reduced norms}\label{subsec:nr}
Let $\mathfrak{o}$ be a noetherian integral domain with field of quotients $F$ and
let $A$ be a finite dimensional semisimple $F$-algebra. If $e_{1}, \ldots, e_{t}$ are the central primitive
idempotents of $A$ then
\[
A = A_{1} \oplus \cdots \oplus A_{t}
\]
where $A_{i}:=Ae_{i}=e_{i}A$.
Each $A_{i}$ is isomorphic to an algebra of $n_{i} \times n_{i}$ matrices over a skewfield $D_{i}$,
and $F_{i}:=\zeta(A_{i})=\zeta(D_{i})$ is a finite field extension of $F$; hence each $A_{i}$ is a central
simple $F_{i}$-algebra. We denote the Schur index of $D_{i}$ by $s_{i}$ so that $[D_{i}:F_{i}]=s_{i}^{2}$.
The reduced norm map
\[
\nr = \nr_{A}: A \longrightarrow \zeta(A)=F_{1} \oplus \cdots \oplus F_{t}
\]
is defined componentwise (see \cite[\S 9]{MR1972204})
and extends to matrix rings over $A$ in the obvious way; hence this induces
a map $K_{1}(A) \rightarrow \zeta(A)^{\times}$ which we also denote by $\nr$.

Now suppose further that $A$ is a separable $F$-algebra and that
$\mathfrak{o}$ is integrally closed. Let $\Lambda$ be an $\mathfrak{o}$-order in $A$.
Then $\Lambda$ is noetherian and so any finitely generated $\Lambda$-module
is in fact finitely presented; we shall use this fact repeatedly without further mention.
By \cite[Corollary (10.4)]{MR1972204} we may choose a maximal order $\Lambda'$
containing $\Lambda$ and there is a decomposition
\[
\Lambda' = \Lambda_{1}' \oplus \cdots \oplus \Lambda_{t}'
\]
where $\Lambda'_{i}=\Lambda' e_{i}$.
Let $\mathfrak{o}_{i}'$ be the integral closure of $\mathfrak{o}$ in $F_{i}$.
Then each $\Lambda_{i}'$ is a maximal $\mathfrak{o}_{i}'$-order with centre $\mathfrak{o}_{i}'$
(see \cite[Theorem (10.5)]{MR1972204}).
A key point is that the reduced norm maps $\Lambda$ into
$\zeta(\Lambda')=\mathfrak{o}_{1}' \oplus \cdots \oplus \mathfrak{o}_{t}'$, but not necessarily into
$\zeta(\Lambda)$.
As above, the reduced norm induces a map $K_{1}(\Lambda) \rightarrow \zeta(\Lambda')^{\times}$
which we again denote by $\nr$.

\begin{remark}\label{rmk:semilocal-nr}
Suppose that $\mathfrak{o}$ is local.
Then $\Lambda$ is semilocal and by \cite[Theorem (40.31)]{MR892316} the natural map
$\Lambda^{\times} \rightarrow K_{1}(\Lambda)$ is surjective.
Furthermore, the diagram
\[
\xymatrix@1@!0@=36pt {
\Lambda^{\times} \ar@{>}[d]_{\nr} \ar@{->}[rr] & & K_{1}(\Lambda) \ar@{>}[dll]^{\nr} \\
\zeta(A) & &
}
\]
commutes and therefore $\nr(\Lambda^{\times})=\nr(K_{1}(\Lambda))=\nr(\Gl_{n}(\Lambda))$ for all $n \in \N$.
\end{remark}

\subsection{Fitting domains and Fitting orders}\label{subsec:fitt-orders}
We shall now specialise to the following situation.
Let $\mathfrak{o}$ be an integrally closed complete commutative noetherian local domain
with field of quotients $F$. We shall refer to $\mathfrak{o}$ as a \emph{Fitting domain}.
For example, one can take $\mathfrak{o}$ to be a complete discrete valuation ring or a power series ring
in one variable over a complete discrete valuation ring.
Let $A$ be a separable $F$-algebra
(i.e.\ a finite dimensional $F$-algebra, such that for every extension field $E$ of $F$, including
$F$ itself, $E \otimes_{F} A$ is a semisimple $E$-algebra; see \cite[\S 7A]{MR632548})
and let $\Lambda$ be an $\mathfrak{o}$-order in $A$.
We shall refer to $\Lambda$ as a \emph{Fitting order} over $\mathfrak{o}$.
A standard example of $\Lambda$ is the group ring $\Z_{p}[G]$ where $p$ is a prime and $G$ is a finite group.

\begin{remark}\label{rmk:grp-ring-sep-alg}
Let $\mathfrak{o}$ be a Fitting domain with field of quotients $F$ and let $G$ be a finite group;
then any $\mathfrak{o}$-order $\Lambda$ in $A:=F[G]$
is a Fitting order if and only if $|G|$ is invertible in $F$.
\end{remark}

\subsection{Reduced norm equivalence}\label{subsec:nre}
We recall the following definition from \cite[\S 1.0.2]{MR2609173}.
Let $N$ and $M$ be two $\zeta(\Lambda)$-submodules of
an $\mathfrak{o}$-torsionfree $\zeta(\Lambda)$-module.
Then $N$ and $M$ are called {\it $\nr(\Lambda)$-equivalent} if
there exists an integer $n$ and a matrix $U \in \Gl_n(\Lambda)$
such that $N = \nr(U) \cdot M$.
(Note that by Remark \ref{rmk:semilocal-nr}, we can in fact replace $\Gl_{n}(\Lambda)$ by $\Lambda^{\times}$
in this definition.)
We say that $N$ is
$\nr(\Lambda)$-contained in $M$ (and write $[N]_{\nr(\Lambda)} \subset [M]_{\nr(\Lambda)}$)
if for all $N' \in [N]_{\nr(\Lambda)}$ there exists $M' \in [M]_{\nr(\Lambda)}$
such that $N' \subset M'$. Note that it suffices to check this property for one $N_0 \in [N]_{\nr(\Lambda)}$.
We will say that $x$ is contained in $[N]_{\nr(\Lambda)}$ (and write $x \in [N]_{\nr(\Lambda)}$) if there is $N_{0} \in [N]_{\nr(\Lambda)}$
such that $x \in N_{0}$.

Let $e \in A$ be a central idempotent.
Suppose that $N$ and $M$ are two $\mathfrak{o}$-torsionfree
$\zeta(\Lambda)$-modules that are $\nr(\Lambda)$-equivalent.
Then $eN$ and $eM$ are $\nr(\Lambda e)$-equivalent $\zeta(\Lambda e)$-modules,
since for $U \in \Lambda^{\times}$ we have $Ue \in (\Lambda e)^{\times}$ and
$\nr_{A}(U) e =\nr_{Ae}(Ue)$. Hence $e[N]_{\nr(\Lambda)} := [eN]_{\nr(\Lambda e)}$
is well-defined.

\subsection{Noncommutative Fitting invariants}\label{subsec:non-comm-fitt}
We recall the following definitions and results from
\cite{MR2609173} and \cite[\S 1.0.3]{nonabstark}.
Let $M$ be a $\Lambda$-module with finite presentation
\begin{equation}\label{eqn:finite_representation}
        \Lambda^{a} \stackrel{h}{\longrightarrow} \Lambda^{b} \onto M.
\end{equation}
We identify the homomorphism $h$ with the corresponding matrix in $M_{a \times b}(\Lambda)$ and define
$S_{b}(h)$ to be the set of all $b \times b$ submatrices of $h$ if $a \geq b$. In the case $a=b$
we call \eqref{eqn:finite_representation} a quadratic presentation.
The Fitting invariant of $h$ over $\Lambda$ is defined to be
\begin{equation}\label{eqn:fitt-def}
\Fitt_{\Lambda}(h) =
\left\{ \begin{array}{lll} [0]_{\nr(\Lambda)} & \mbox{ if } & a<b \\
\left[\langle \nr(H) \mid H \in S_{b}(h)\rangle_{\zeta(\Lambda)}\right]_{\nr(\Lambda)} & \mbox{ if } & a \geq b.
\end{array}\right.
\end{equation}
We call $\Fitt_{\Lambda}(h)$ a Fitting invariant of $M$ over $\Lambda$.
If $M$ admits a quadratic presentation $h$ we put $\Fitt_{\Lambda}(M) := \Fitt_{\Lambda}(h)$,
which can be shown to be independent of the chosen quadratic presentation
(see \cite[Theorem 3.2]{MR2609173}).
We define $\Fitt_{\Lambda}^{\max}(M)$ to be the unique
Fitting invariant of $M$ over $\Lambda$ which is maximal
among all Fitting invariants of $M$ with respect to
the partial order ``$\subset$''. Finally, we define a $\zeta(\Lambda)$-submodule of $\zeta(A)$ by
\[
\mathcal{I} = \mathcal{I}(\Lambda) := \langle \nr(H) \mid H \in M_{b \times b}(\Lambda), b \in \N  \rangle_{\zeta(\Lambda)}
\]
and note that this is in fact an $\mathfrak{o}$-order in $\zeta(A)$ contained in $\zeta(\Lambda')$.

\begin{theorem}\label{thm:fitt-thm}
Let $M,M_{1}, M_{2}$ and $M_{3}$ be finitely generated $\Lambda$-modules.
\begin{enumerate}
\item If $M_{1} \twoheadrightarrow M_{2}$ is an epimorphism then
$\Fitt_{\Lambda}^{\max}(M_{1}) \subset \Fitt_{\Lambda}^{\max}(M_{2})$.
\item If $M_{1} \rightarrow M_{2} \twoheadrightarrow M_{3}$ is an exact sequence, then
\[
\Fitt_{\Lambda}^{\max}(M_{1}) \cdot \Fitt_{\Lambda}^{\max}(M_{3}) \subset \Fitt_{\Lambda}^{\max}(M_{2}).
\]
\item Let $M_{1} \hookrightarrow M_{2} \twoheadrightarrow M_{3}$ be an exact sequence.
If $M_{1}$ and $M_{3}$ admit quadratic presentations, so does $M_{2}$ and
\[
\Fitt_{\Lambda}(M_{1}) \cdot \Fitt_{\Lambda}(M_{3}) = \Fitt_{\Lambda}(M_{2}).
\]
\item If $\theta \in \Fitt_{\Lambda}^{\max}(M)$ and $\lambda \in \mathcal{I}$ then
$\lambda \cdot \theta \in \Fitt_{\Lambda}^{\max}(M)$.
\item If $M$ admits a quadratic presentation, then $\Fitt_{\Lambda}^{\max}(M) =  \mathcal{I} \cdot \Fitt_{\Lambda}(M)$.
\item Let $e \in A$ be a central idempotent. Then
$e \Fitt_{\Lambda}^{\max}(M) = \Fitt_{\Lambda e}^{\max}(\Lambda e \otimes_{\Lambda} M)$.
\item Set $M_{F} :=  F \otimes_{\mathfrak{o}} M$ and
$\Upsilon(M) := \{ i \in \{1, \ldots, t \} \mid e_{i}M_{F} =0 \}$.
Then
\[
\Fitt_{\Lambda}^{\max}(M) = e \Fitt_{\Lambda}^{\max}(M) = \Fitt_{\Lambda e}^{\max}(\Lambda e \otimes_{\Lambda} M)
\]
where $e=e(M) := \sum_{i \in \Upsilon(M)} e_{i}$.
\end{enumerate}
\end{theorem}

\begin{proof}
For (i), (ii) and (iii), see \cite[Proposition 3.5]{MR2609173}.
For (iv) and (v) see \cite[Proposition 1.1]{nonabstark}.
For (vi) and (vii) see \cite[Lemma 3.4]{MR2609173}.
\end{proof}

\subsection{An alternative definition of noncommutative Fitting invariants}\label{subsec:alt-def}
We define
\[
\mathcal{U} = \mathcal{U}(\Lambda) := \langle \nr(H) \mid H \in \Gl_{b}(\Lambda), b \in \N  \rangle_{\zeta(\Lambda)} = \langle \nr(H) \mid H \in \Lambda^{\times}  \rangle_{\zeta(\Lambda)},
\]
where the last equality is due to Remark \ref{rmk:semilocal-nr}.
This is an $\mathfrak{o}$-order in $\zeta(A)$ contained in $\mathcal{I}(\Lambda)$.
Let $M$ be a $\Lambda$-module with finite presentation
\[
\Lambda^{a} \stackrel{h}{\longrightarrow} \Lambda^{b} \onto M.
\]
An alternative definition to \eqref{eqn:fitt-def} is
\begin{equation}\label{eqn:fit-def}
\Fit_{\Lambda}(h) =
\left\{ \begin{array}{lll} \langle 0 \rangle_{\mathcal{U}(\Lambda)} & \mbox{ if } & a<b \\
\langle \nr(H) \mid H \in S_{b}(h)\rangle_{\mathcal{U}(\Lambda)} & \mbox{ if } & a \geq b.
\end{array}\right.
\end{equation}
(Note that $\Fitt_{\Lambda}(h)$ of \eqref{eqn:fitt-def} has two t's whereas $\Fit_{\Lambda}(h)$ of \eqref{eqn:fit-def}
has one t.)
We define $\Fit_{\Lambda}^{\max}(M)$ to be the unique
Fitting invariant of $M$ over $\Lambda$ which is maximal with respect to
inclusion among all $\Fit_{\Lambda}(h')$ where $h'$ is a presentation of $M$.
An argument analogous to that given for Theorem \ref{thm:fitt-thm}(iv) shows that
$\Fit_{\Lambda}^{\max}(M)$ is in fact a module over $\mathcal{I}(\Lambda)$.

The two definitions are explicitly related as follows.
Consider the category $\mathcal{N}$ with $\nr(\Lambda)$-equivalence classes of finitely generated
$\zeta(\Lambda)$-submodules of $\zeta(A)$ as objects and inclusions as morphisms.
Let $\mathcal{M}$ be the category of finitely generated $\mathcal{I}(\Lambda)$-submodules
of $\zeta(A)$ with inclusions as morphisms. Then
\begin{eqnarray} \label{eqn:iota-functor}
    \iota: \mathcal{N} & \longrightarrow & \mathcal{M} \\
    {[X]_{\nr(\Lambda)}} & {\mapsto} & {X \cdot \mathcal{I}(\Lambda)} \nonumber
\end{eqnarray}
is a covariant functor. Note that $\iota$ is well-defined:
If $X'$ is $\nr(\Lambda)$-equivalent to $X$, then there is  a
$U \in \Lambda^{\times}$ such that $X' = \nr(U) \cdot X$;
but $\nr(U) \in \mathcal{I}(\Lambda)^{\times}$ and hence
$X' \cdot \mathcal{I}(\Lambda) = X \cdot \mathcal{I}(\Lambda)$.
In the special case $\zeta(\Lambda)=\mathcal{I}(\Lambda)$ (e.g.\ $\Lambda$ is commutative or maximal),
the equivalence class
$[X]_{\nr(\Lambda)}$ contains precisely one element and we have $\iota([X]_{\nr(\Lambda)})=X$.
In the general case, it is straightforward to see that we have
\begin{equation} \label{eqn:Fitt-vs-Fit}
\iota(\Fitt^{\max}_{\Lambda}(M)) = \Fit^{\max}_{\Lambda}(M).
\end{equation}
It follows that $\Fit_{\Lambda}^{\max}(M)$ has the
properties analogous to those of $\Fitt_{\Lambda}^{\max}(M)$ given in Theorem \ref{thm:fitt-thm}.

The advantage of $\Fit^{\max}_{\Lambda}(M)$ is that $\nr(\Lambda)$-equivalence classes
are not required and, as we shall see, it is compatible with Definition \ref{def:fitt-matrix};
the advantage of $\Fitt^{\max}_{\Lambda}(M)$ is that it can be directly related to Fitting invariants of quadratic presentations
which in turn can be used to do computations in relative $K$-groups. For instance, the application in \cite[{\S}7]{MR2609173} shows
how to compute annihilators of the class group of a number field via this notion
of Fitting invariants
from an appropriate special case of the equivariant Tamagawa number conjecture (which asserts a certain equality in a relative $K$-group).
Moreover, it can be used to define relative Fitting invariants (see \cite[p.2764]{MR2609173}).
However, in most cases
it does not really matter which definition we work with, as they are explicitly related as above.
For the rest of this article, the reader may almost always think in terms of
$\Fit^{\max}_{\Lambda}(M)$ rather than $\Fitt_{\Lambda}^{\max}(M)$.

\subsection{Generalised adjoint matrices}\label{subsec:adjoints}
Choose $n \in \N$ and let $H \in M_{n \times n}(\Lambda)$.
Then recalling the notation of \S \ref{subsec:nr},
decompose $H$ into
\[
H = \sum_{i=1}^{t} H_{i} \in M_{n \times n}(\Lambda') = \bigoplus_{i=1}^t  M_{n \times n}(\Lambda'_{i}),
\]
where $H_{i}:=He_{i}$.
Let $m_{i} = n_{i} \cdot s_{i} \cdot n$.
The reduced characteristic polynomial $f_{i}(X) = \sum_{j=0}^{m_{i}} \alpha_{ij}X^{j}$ of $H_{i}$
has coefficients in $\mathfrak{o}_{i}'$.
Moreover, the constant term $\alpha_{i0}$ is equal to $\nr(H_{i}) \cdot (-1)^{m_{i}}$.
We put
\[
H_{i}^{\ast} := (-1)^{m_{i}+1} \cdot \sum_{j=1}^{m_i} \alpha_{ij}H_{i}^{j-1}, \quad H^{\ast} := \sum_{i=1}^{t} H_{i}^{\ast}.
\]

\begin{lemma}\label{lem:ast}
We have $H^{\ast} \in M_{n\times n} (\Lambda')$ and $H^{\ast} H = H H^{\ast} = \nr_A(H) \cdot 1_{n \times n}$.
\end{lemma}

\begin{proof}
The first assertion is clear by the above considerations.
Since $f_{i}(H_{i}) = 0$, we find that
\[
H_{i}^{\ast} \cdot H_{i} = H_{i} \cdot H_{i}^{\ast}  = (-1)^{m_{i}+1} (-\alpha_{i0}) = \nr(H_{i}),
\]
as desired.
\end{proof}

\begin{remark}
Note that the above definition of $H^{\ast}$ differs slightly from the definition in \cite[\S 4]{MR2609173}.
However, the only properties of $H^{\ast}$ needed are those stated in Lemma \ref{lem:ast}.
Moreover, if $H$ is invertible (over $A$), then $H^{\ast}$ is uniquely determined by the
equation in Lemma \ref{lem:ast}, and hence the two definitions agree in this case.
The new definition has the advantage that it is precisely the adjoint matrix if $\Lambda$ is commutative,
and the assignment $H \mapsto H^{\ast}$ is often continuous (e.g.~with respect to the $p$-adic
topology if $\mathfrak{o} = \Z_{p}$).
\end{remark}

We define
\[
\mathcal{H} = \mathcal{H}(\Lambda) := \{ x \in \zeta(\Lambda) \mid xH^{\ast} \in M_{b \times b}(\Lambda) \, \forall H \in M_{b \times b}(\Lambda) \, \forall b \in \N \}.
\]
Since $x \cdot \nr(H) = xH^{\ast}H \in \zeta(\Lambda)$, in particular we have
$\mathcal{H} \cdot \mathcal{I} = \mathcal{H} \subset \zeta(\Lambda)$. Hence $\mathcal{H}$ is an ideal in the $\mathfrak{o}$-order
$\mathcal{I}(\Lambda)$.

\subsection{Fitting invariants and annihilation}

\begin{theorem}\label{thm:fitt-ann}
Let $\Lambda$ be a Fitting order and let $M$ be a finitely generated $\Lambda$-module. Then
\[
\mathcal{H}(\Lambda) \cdot \Fit_{\Lambda}^{\max}(M) \subset \Ann_{\zeta(\Lambda)}(M).
\]
\end{theorem}

\begin{proof}(Also see \cite[Theorem 4.2]{MR2609173}.)
Let $\Lambda^{a} \stackrel{h}{\longrightarrow} \Lambda^{b} \onto M$ be a finite presentation of $M$.
Then it suffices to show that $\mathcal{H}(\Lambda) \cdot \Fit_{\Lambda}(h) \subset \Ann_{\zeta(\Lambda)}(M)$.
Fix $H \in S_{b}(h)$ and $x \in \mathcal{H}(\Lambda)$.
As $\Fit_{\Lambda}(h)$ is generated by elements of
the form $\nr(H)$, we are further reduced to showing that $x \cdot \nr(H)$ annihilates $M$.
The cokernel of $H$ surjects onto $M$ and hence the assertion follows from the commutative
diagram
\[
\xymatrix@1@!0@=36pt {
\Lambda^{b} \ar@{>}[rr]^{H}  & & \Lambda^{b} \ar@{>}[d]^{x \cdot \nr(H)} \ar@{>}[lld]_{x \cdot H^{\ast}}  \ar@{>>}[rr] & & \coker(H)
 \ar@{>}[d]^{x \cdot \nr(H)} \\
\Lambda^{b} \ar@{>}[rr]^{H} & & \Lambda^{b}  \ar@{>>}[rr] & & \coker(H)
}
\]
once one notes that the right most map is zero.
\end{proof}

\subsection{Fitting invariants of matrix rings over commutative rings}
Fix $n \in \N$ and let $\Lambda=M_{n \times n}(R)$
where $R$ is a commutative $\mathfrak{o}$-order.
Hence $\Lambda$ is both a Fitting order and a matrix ring over a commutative ring.
The aim of this section is to show that Definition \ref{def:fitt-matrix} is compatible with \eqref{eqn:fit-def}
in this case, thereby justifying the similar notation.

\begin{prop}\label{prop:matrix-over-comm-consistent}
Let $M$ be a finitely generated $\Lambda$-module. Then $\Fit_{\Lambda}(M) = \Fit^{\max}_{\Lambda}(M)$.
\end{prop}

\begin{proof}
First note that $R=\zeta(\Lambda)=\mathcal{U}(\Lambda)=\mathcal{I}(\Lambda)$.
Let
$
\Lambda^{a} \stackrel{h}{\longrightarrow} \Lambda^{b} \onto M
$
be a finite presentation of $\Lambda$. We can and do assume without loss of generality that $a \geq b$.
Let $H \in M_{a \times b}(\Lambda)$ and $H',\tilde{H}\in M_{na \times nb}(R)$
be the matrices corresponding to $h$ as in Lemma \ref{lem:lambda-matrix-to-R}; then
$H'=\tilde{H}$.
Hence we have
\begin{eqnarray*}
\Fit_{\Lambda}(h) : = \langle \nr(T) \mid T \in S_{b}(H) \rangle_{R}
&\subset&  \langle \nr(\tilde{T}) \mid \tilde{T} \in S_{nb}(\tilde{H}) \rangle_{R} \\
&=& \langle \nr(\tilde{T}) \mid \tilde{T} \in S_{nb}(H') \rangle_{R} \\
&=& \Fit_{R}(e_{11}M) =: \Fit_{\Lambda}(M).
\end{eqnarray*}
It follows that $\Fit^{\max}_{\Lambda}(M) \subset \Fit_{\Lambda}(M)$.

Now let $\tilde{T} \in S_{nb}(H')$. Then by swapping rows of $H'$ appropriately,
there exists $\tilde{E} \in \Gl_{na}(R)$
with $\det_{R}(\tilde{E}) = \pm 1$ such that the $nb \times nb$ submatrix of
$\tilde{E}H'$ formed by taking the first $nb$ rows is equal to $\tilde{T}$.
Let $E \in M_{a \times a}(\Lambda)$ (resp.\ $T \in M_{b \times b}(\Lambda)$) be the same matrix as $\tilde{E}$
(resp.\ $\tilde{T}$) but with entries considered in $\Lambda$ rather than $R$.
Then $E \in \Gl_{a}(\Lambda)$ and the diagram
\[
\xymatrix@1@!0@=36pt {
\Lambda^{a} \ar@{>}[rr]^{EH} \ar@{>}[d]^{E}_{\simeq}  & & \Lambda^{b} \ar@{=}[d] \ar@{>>}[rr] & & \coker(EH)
 \ar@{>}[d]^{\simeq} \\
\Lambda^{a} \ar@{>}[rr]^{H} & & \Lambda^{b} \ar@{>>}[rr] & & M
}
\]
commutes. (Note that the order of function composition and corresponding matrix multiplication
are reversed since we consider left $\Lambda$-modules and so functions are represented by
multiplying by their corresponding matrices on the right.)
Since $T$ is a $b \times b$ submatrix of $EH$ we therefore have
\[
\nr(\tilde{T})=\nr(T) \in \langle \nr(V) \mid V \in S_{b}(EH) \rangle_{R}
\subset \Fit^{\max}_{\Lambda}(\coker(EH)) = \Fit_{\Lambda}^{\max}(M).
\]
Since $\tilde{T} \in S_{nb}(H')$ was arbitrary, we have shown that
\[
\Fit_{\Lambda}(M) := \Fit_{R}(e_{11}M)
= \langle \nr(\tilde{V}) \mid \tilde{V} \in S_{nb}(H') \rangle_{R}
\subset \Fit_{\Lambda}^{\max}(M).
\]
Therefore we have  $\Fit^{\max}_{\Lambda}(M) = \Fit_{\Lambda}(M)$, as required.
\end{proof}

\section{Nice Fitting orders}\label{sec:nice-fItting-orders}

\begin{definition}\label{def:nice}
Let $\Lambda$ be a Fitting order over $\mathfrak{o}$.
Suppose that
$\Lambda = \oplus_{j=1}^{k} \Lambda_{j}$
where each $\Lambda_{j}$ is either a maximal $\mathfrak{o}$-order or is of the form $M_{a_{j} \times a_{j}}(\Gamma_{j})$
for some commutative ring $\Gamma_{j}$.
Then we say that $\Lambda$ is a \emph{nice} Fitting order.
\end{definition}

\begin{remark}
If a Fitting order $\Lambda$ is either maximal or commutative then it is immediate from the definition that $\Lambda$ is nice.
\end{remark}

\begin{prop}\label{prop:nice-equal}
Let $\Lambda$ be a nice Fitting order.
Then $\mathcal{U}(\Lambda)=\mathcal{I}(\Lambda)=\mathcal{H}(\Lambda)=\zeta(\Lambda)$.
\end{prop}

\begin{proof}
Fix $n \in \N$ and let $H \in M_{n \times n}(\Lambda)$.
Write $H=\sum_{j=1}^{k} H_{j}$ corresponding to the decomposition $\Lambda = \oplus_{j=1}^{k} \Lambda_{j}$.
If $\Lambda_{j}$ is a maximal order then it is clear from the definition of
$H_{j}^{*}$ that $H_{j}^{*} \in  M_{n \times n}(\Lambda_{j})$.
If $\Lambda_{j} \simeq M_{a_{j} \times a_{j}}(\Gamma_{j})$
for some commutative ring $\Gamma_{j}$, then $H_{j}^{\ast}$ is the usual adjoint matrix if considered as a matrix
in  $M_{na_{j} \times na_{j}}(\Gamma_{j})$, and so $H_{j}^{\ast} \in M_{n \times n}(\Lambda_{j})$.
Therefore $H^{\ast}=\sum_{j=1}^{k}H_{j}^{*}$ lies in $M_{n \times n}(\Lambda)$.
Since $n$ was arbitrary, it follows that $\zeta(\Lambda) \subset \mathcal{H}(\Lambda)$.
In particular, $1 \in \mathcal{H}(\Lambda)$ so must have $\mathcal{H}(\Lambda)=\mathcal{I}(\Lambda)$
since $\mathcal{H}(\Lambda)$ is an ideal of $\mathcal{I}(\Lambda)$.
Thus $\zeta(\Lambda)=\mathcal{I}(\Lambda)=\mathcal{H}(\Lambda)$.
The desired result now follows from the inclusions $\zeta(\Lambda) \subset \mathcal{U}(\Lambda) \subset \mathcal{I}(\Lambda)$.
\end{proof}

\begin{corollary} \label{cor:intersection-nice-Fitt-orders}
Suppose $\Lambda$ is a Fitting order that is an intersection of nice Fitting orders or is such that $\zeta(\Lambda)$
is maximal.
Then $\mathcal{U}(\Lambda)=\mathcal{I}(\Lambda)=\mathcal{H}(\Lambda)=\zeta(\Lambda)$.
In particular, this is the case if $\Lambda$ is a hereditary or graduated order over a complete discrete valuation ring.
\end{corollary}

\begin{proof}
Suppose $\Lambda = \cap_{i} \Lambda_{i}$ where each $\Lambda_{i}$ is a nice Fitting order.
Fix $n \in \N$ and let $H \in M_{n \times n}(\Lambda)$. Then the argument above shows that
$H^{\ast} \in \Lambda_{i}$ for each $i$ and so $H^{\ast} \in \Lambda$. The rest of the argument follows as before.
If $\zeta(\Lambda)$ is maximal, then the result follows directly from the definitions in \S \ref{subsec:adjoints}.

Let $\Lambda$ be a graduated order over a complete discrete valuation ring.
(Recall that an order is graduated if there exist orthogonal primitive idempotents $e_{1}, \ldots, e_{t} \in \Lambda$
with $1 = e_{1} + \cdots + e_{t}$ such that $e_{i}\Lambda e_{i}$ is a maximal order for $i=1,\ldots,t$. In particular,
maximal and hereditary orders are graduated. See \cite[\S II]{MR724074} for further details.)
The result now follows from the observation that $\zeta(\Lambda)$ is maximal.
\end{proof}

\begin{definition}\label{def:max-comm-hybrid}
Let $G$ be a finite group with  commutator subgroup $G'$.
Let $\mathfrak{o}$ be a Fitting domain whose characteristic is either zero or does not divide $|G|$.
Let $\Lambda'$ be a maximal order containing the group ring $\mathfrak{o}[G]$
and let $e=|G'|^{-1}\Tr_{G'}$ where $\Tr_{G'} := \sum_{g' \in G'}g'$.
Define $\Lambda'(G,G'):=\mathfrak{o}[G]e \oplus \Lambda'(1-e)$.
\end{definition}

\begin{prop}\label{prop:max-comm-hybrid}
In the setting above, $\Lambda'(G,G')$ is a nice Fitting order containing $\mathfrak{o}[G]$.
\end{prop}

\begin{proof}
Remark \ref{rmk:grp-ring-sep-alg} and the hypotheses ensure that $\Lambda'(G,G')$ is in fact a Fitting order.
Note that $\mathfrak{o}[G]e$ is commutative and $\Lambda'(1-e)$ is maximal;
hence $\Lambda'(G,G')$ is nice.
The second assertion follows from the observation that
$\Lambda'(G,G') = \mathfrak{o}[G] + \Lambda'(1-e)$.
\end{proof}

\begin{remark}
Of course, $\Lambda'(G,G')$ depends on the choice of $\Lambda'$.
However, for many applications this choice does not matter.
For explicit examples, see Examples \ref{ex:A4} and \ref{ex:D8}.
\end{remark}

\begin{prop}\label{prop:p-not-div-comm-finite}
Let $\mathfrak{o}$ be a Fitting domain with residue field of characteristic $p>0$
and let $G$ be a finite group with  commutator subgroup $G'$.
Suppose the characteristic of $\mathfrak{o}$ is either zero or does not divide $|G|$.
Then the group ring $\mathfrak{o}[G]$ is a Fitting order and the following are equivalent:
\begin{enumerate}
\item $p$ does not divide $|G'|$;
\item $\mathfrak{o}[G]$ is a direct sum of matrix rings over commutative rings;
\item $\mathfrak{o}[G]$ is a nice Fitting order;
\item $\mathcal{H}(\mathfrak{o}[G])=\zeta(\mathfrak{o}[G])$;
\item the group $G$ has an abelian $p$-Sylow subgroup $P$ and a normal $p$-complement $N$,
in which case $G$ is isomorphic to a semi-direct product $N \rtimes P$.
\end{enumerate}
\end{prop}

\begin{proof}
Remark \ref{rmk:grp-ring-sep-alg} and the hypotheses ensure that $\mathfrak{o}[G]$ is in fact a Fitting order.
The equivalence of (i) and (v) is 
a straightforward exercise in elementary group theory (also see \cite[p.\ 390]{MR704622}).
A special case of  \cite[Corollary, p.\ 390]{MR704622} shows that (i) implies (ii); 
Definition \ref{def:nice} immediately gives (ii) $\implies$ (iii); 
Proposition \ref{prop:nice-equal} gives (iii) $\implies$ (iv).
It remains to show that (iv) $\implies$ (i).
Set $\Lambda := \mathfrak{o}[G]$ and suppose that  $\mathcal{H}(\Lambda)=\zeta(\Lambda)$.
Let $H=0 \in \Lambda = M_{1 \times 1}(\Lambda)$.
Recall the notation of \S \ref{subsec:adjoints} and write
$H = \sum_{i=1}^{t} H_{i} \in \oplus_{i=1}^{t} \Lambda_{i}'$.
Then the reduced characteristic polynomial of $H_{i}$ is $f_{i}(X)=X^{n_{i}s_{i}}$
and so $H_{i}^{\ast}$ is $h_{i}(0)$ where $h_{i}(X):=X^{n_{i}s_{i}-1}$.
Hence $H_{i}^{\ast}=1$ if $n_{i}s_{i}=1$ and $H_{i}^{\ast}=0$ if $n_{i}s_{i}>1$.
Therefore $H^{\ast}=|G'|^{-1}\Tr_{G'}$.
However, $1 \in \zeta(\Lambda)=\mathcal{H}(\Lambda)$ and so 
$H^{\ast} \in \Lambda=\mathfrak{o}[G]$ by definition of $\mathcal{H}(\Lambda)$ (see \S \ref{subsec:adjoints}).
But then $|G'|$ must be invertible in $\mathfrak{o}$
and so $p \nmid |G'|$ since the residue field of $\mathfrak{o}$ has characteristic $p$.
\end{proof}

\begin{example}\label{ex:A4}
Let $A_{4}$ be the alternating group on $4$ letters.
Then $\Z_{3}[A_{4}]$ is neither commutative nor maximal, yet is
a nice Fitting order by an application of Proposition \ref{prop:p-not-div-comm-finite}.
In fact, one can show that $\Z_{3}[A_{4}] = \Lambda'(A_{4},A_{4}')$ (recall Definition \ref{def:max-comm-hybrid}) 
where $\Lambda'$ is the unique maximal order in $\Q_{3}[A_{4}]$ containing $\Z_{3}[A_{4}]$ and 
$A_{4}'$ is the commutator subgroup of $A_{4}$.
\end{example}

\begin{example}\label{ex:metacyclic}
Let $p,q$ be distinct primes with $p$ odd such that $q \mid (p-1)$.
Let $r$ be a primitive $q$-th root of $1$ $\mathrm{mod}$ $p$.
Let $F_{p,q}:= \langle x, y \mid x^{p}=y^{q}=1, yxy^{-1}=y^{r} \rangle$.
Then $F_{p,q}$ is a metacyclic group of order $pq$ and in the special
case $q=2$, this is the dihedral group of order $2p$.
One can show that $\Z_{q}[F_{p,q}]$ is a nice Fitting order
by either applying Proposition \ref{prop:p-not-div-comm-finite} or
following the explicit computation of \cite[\S 34E]{MR632548}.
\end{example}

\begin{remark}\label{rmk:non-ab-stick-improvement}
Let $L/K$ be a finite Galois CM-extension of number fields with Galois group $G$.
Let $p$ be an odd prime and let $\cl_{L}$ denote the class group of $L$.
Under mild technical hypotheses on $p$, \cite[Theorem 1.2]{MR2771125}
gives annihilators of $\Z_{p} \otimes_{\Z} \cl_{L}$ in terms of special values of a truncated Artin L-function of $L/K$.
Building on this result, \cite[Corollary 7.2]{MR2609173} uses noncommutative Fitting
invariants to predict similar annihilators
under the assumption of the relevant special case of the $p$-part of the
equivariant Tamagawa number conjecture (ETNC) (see \cite{MR1884523}, \cite{MR1863302}).
Now Proposition \ref{prop:p-not-div-comm-finite}
can be used to give explicit examples in which \cite[Corollary 7.2]{MR2609173} predicts strictly
more annihilators than the unconditional annihilators of \cite[Theorem 1.2]{MR2771125}
(e.g.\ one can use a minor variant of Example \ref{ex:A4} in the case
$p=3$ and $G=A_{4} \times C_{2}$, where $C_{2}$ is the group of order $2$.)
The results of \S \ref{sec:ann-and-change-of-order} can be used to give further examples
in the case that $p$ divides $|G'|$.
\end{remark}

\begin{prop}\label{prop:p-not-div-comm-profinite}
Let $\mathfrak{o}$ be a Fitting domain of characteristic zero with
residue field of characteristic $p>0$.
Let $G$ be a profinite group containing a finite normal subgroup $H$ such
that $G/H \simeq \Gamma$, where $\Gamma$ is a pro-$p$ group isomorphic to $\Z_{p}$.
Then the commutator subgroup $G'$ is finite, the complete group algebra $\mathfrak{o}[[ G ]]$
is a Fitting order, and the following are equivalent:
\begin{enumerate}
\item $p$ does not divide $|G'|$;
\item $\mathfrak{o}[[G]]$ is a direct sum of matrix rings over commutative rings;
\item $\mathfrak{o}[[G]]$ is a nice Fitting order;
\item $\mathcal{H}(\mathfrak{o}[[G]])=\zeta(\mathfrak{o}[[G]])$.
\end{enumerate}
\end{prop}

\begin{remark}
If the characteristic of $\mathfrak{o}$ is non-zero then $\mathfrak{o}[[G]]$ is not necessarily separable and
thus not necessarily a Fitting order. However, even in this situation the proof below shows it is still the case that (i) $\implies$ (ii)
and so the results of Theorem \ref{thm:properties-matrix-over-comm} can still be applied when (i) holds.
A similar remark also applies to Proposition \ref{prop:p-not-div-comm-finite}.
\end{remark}

\begin{proof}
Set $\Lambda:=\mathfrak{o}[[G]]$ and 
let $\mathfrak{O} := \mathfrak{o}[[ T ]]$
be the power series ring in one variable over $\mathfrak{o}$.
We fix a topological generator $\gamma$ of $\Gamma$ and choose a natural number $n$ such that
$\gamma^{p^{n}}$ is central in $G$.
Since $\Gamma^{p^n} \simeq \Z_{p}$, there is an isomorphism
$\mathfrak{o}[[ \Gamma^{p^n} ]] \simeq \mathfrak{O}$
induced by $\gamma^{p^n} \mapsto 1+T$.
Note that $G$ can be written as a semi-direct product $H \rtimes \Gamma$;
hence if we view $\Lambda$ as an $\mathfrak{O}$-module, there is a decomposition
\[
\Lambda = \bigoplus_{i=0}^{p^n-1} \mathfrak{O} \gamma^{i} [H].
\]
Hence $\Lambda$ is finitely generated as an $\mathfrak{O}$-module and is an
$\mathfrak{O}$-order in the separable $F:=Quot(\mathfrak{O})$-algebra
$A = \mathcal{Q}(G) := \bigoplus_{i} F \gamma^i[H]$.
Note that $A$ is obtained from $\Lambda$ by inverting all regular elements.
Since $\mathfrak{O}$ is again a Fitting domain, $\Lambda$ is a Fitting order over $\mathfrak{O}$.

Let $\mathfrak{p}$ (resp.\ $\mathfrak{P}$) be the maximal ideal of $\mathfrak{o}$ (resp.\ $\mathfrak{O}$).
Then $\mathfrak{P}$ is generated by $\mathfrak{p}$ and $T$.
Since $\gamma^{p^n} = 1 + T \equiv 1 \mod \mathfrak{P}$, we have
\[
\overline{\Lambda} := \Lambda / \mathfrak{P} \Lambda = \bigoplus_{i=0}^{p^{n}-1} k \gamma^{i} [H]
= k[H \rtimes C_{p^n}],
\]
where $C_{p^n}$ denotes the cyclic group of order $p^n$ and
$k := \mathfrak{O} / \mathfrak{P} = \mathfrak{o}/\mathfrak{p}$ is the residue field of characteristic $p$.
Since $G / H$ is abelian, the commutator subgroup $G'$ of $G$ is actually a subgroup of $H$
and thus is finite. Moreover, $G'$ identifies with the commutator subgroup of $H \rtimes C_{p^n}$.

That (ii) $\implies$ (iii) $\implies$ (iv) $\implies$ (i) follows by the same reasoning as that in the proof of
Proposition \ref{prop:p-not-div-comm-finite}. Thus it remains to show that (i) $\implies$ (ii). 
Suppose that $p \nmid |G'|$.
Then $k[H \rtimes C_{p^n}]$ is separable over its centre, i.e., it is an Azumaya algebra by \cite[Theorem 1]{MR704622}.
Moreover, $\overline{\Lambda}=k[H \rtimes C_{p^n}]$ and $k=\mathfrak{O}/\mathfrak{P}$,
so \cite[Theorem 4.7]{MR0121392} 
shows that $\Lambda$ is also an Azumaya algebra.
However, $\zeta(\Lambda)$ is semiperfect by \cite[Example 23.3]{MR1838439}
and thus a direct sum of local rings by \cite[Theorem 23.11]{MR1838439},
say
\[
\zeta(\Lambda) = \bigoplus_{i=1}^{r} \mathfrak{O}_{i},
\]
where each $\mathfrak{O}_{i}$ contains $\mathfrak{O}$.
By \cite[Proposition 6.5 (ii)]{MR632548} each $ \mathfrak{O}_{i}$ is in fact a complete local ring.
Let $\mathfrak{P}_{i}$ be the maximal ideal of $\mathfrak{O}_{i}$ and
$k_{i} := \mathfrak{O}_{i} / \mathfrak{P}_{i}$ be the residue field.
Since $\mathfrak{P} \subset \mathfrak{P}_{i}$,
the natural projection $\mathfrak{O}_{i} \onto k_{i}$ factors through
$\mathfrak{O}_{i} \onto \mathfrak{O}_{i} / \mathfrak{P} = \mathfrak{O}_{i} \otimes_{\mathfrak{O}} k$.
Hence we have the corresponding homomorphisms of Brauer groups
\[
\Br(\mathfrak{O}_{i}) \rightarrow \Br(\mathfrak{O}_{i} / \mathfrak{P}) \rightarrow \Br(k_{i}).
\]
Now $\Br(\mathfrak{O}_{i}) \to \Br(k_{i})$ is injective by
\cite[Corollary 6.2]{MR0121392}
and hence $\Br(\mathfrak{O}_{i}) \rightarrow \Br(\mathfrak{O}_{i} / \mathfrak{P})$ must also be injective.
This yields an embedding
\[
\Br(\zeta(\Lambda)) = \bigoplus_{i=1}^{r} \Br(\mathfrak{O}_{i}) \hookrightarrow \bigoplus_{i=1}^{r} \Br(\mathfrak{O}_{i} \otimes_{\mathfrak{O}} k) = \Br(\zeta(\Lambda) \otimes_{\mathfrak{O}} k).
\]
Since $\Lambda$ is Azumaya, it defines a class $[\Lambda] \in \Br(\zeta(\Lambda))$ which is mapped to
$[\overline{\Lambda}]$ via this embedding.
However, $\overline{\Lambda}$ is a group ring of a finite group over a field of positive characteristic
and, as noted in the remark after \cite[Corollary, p.\ 390]{MR704622}, 
such a group ring is Azumaya if and only if it is a direct product of matrix rings over commutative rings.
Hence $[\overline{\Lambda}]$ is trivial and thus so is $[\Lambda]$.
Therefore $\Lambda$ is a direct sum of matrix rings over commutative rings.
\end{proof}

\begin{theorem}\label{thm:nice-properties}
Let $\Lambda$ be a nice Fitting order over the Fitting domain $\mathfrak{o}$.
Let $M,M_{1}, M_{2}$ and $M_{3}$ be finitely generated $\Lambda$-modules.
\begin{enumerate}
  \item We have $\Fit_{\Lambda}^{\max}(M) \subset \Ann_{\zeta(\Lambda)}(M)$.
  \item Suppose that $\Lambda$ is a direct sum of matrix rings over commutative rings or that $\mathfrak{o}$ is a complete discrete valuation ring. If $M_{2} = M_{1} \oplus M_{3}$, then
 \[
\Fit^{\max}_{\Lambda}(M_{2}) =  \Fit^{\max}_{\Lambda}(M_{1}) \cdot \Fit^{\max}_{\Lambda}(M_{3}).
 \]
  \item If $\Lambda$ is a maximal order over a complete discrete valuation ring $\mathfrak{o}$,
  and $M_{1} \hookrightarrow M_{2} \onto M_{3}$ is an exact sequence, then
  \begin{equation} \label{eqn:additivity-in-ses}
      \Fit^{\max}_{\Lambda}(M_{2}) =  \Fit^{\max}_{\Lambda}(M_{1}) \cdot \Fit^{\max}_{\Lambda}(M_{3}).
  \end{equation}
\end{enumerate}
\end{theorem}

\begin{proof}
Property (i) follows from combining Proposition \ref{prop:nice-equal} and Theorem \ref{thm:fitt-ann}.
For (ii) it suffices to treat the cases where $\Lambda$ is a matrix ring over a commutative ring or a maximal order over a complete discrete valuation ring.
In the former case, (ii) is Theorem \ref{thm:properties-matrix-over-comm} (iv); in the latter, (ii) follows from (iii) applied to the tautological exact sequence
$M_{1} \hookrightarrow M_{1} \oplus M_{3} \onto M_{3}$. So it suffices to prove (iii).
We shall need the following lemma.

\begin{lemma}\label{lem:dvr-implies-quadratic}
Let $\Lambda$ be a maximal order over a complete discrete valuation ring $\mathfrak{o}$ such that the $F$-algebra $A$ is simple. Let $M$ be a finitely generated $\Lambda$-module.
Then either $F \otimes_{\mathfrak{o}} M \neq 0$ and $\Fit^{\max}_{\Lambda}(M) = 0$ or $M$ admits a quadratic presentation.
\end{lemma}

\begin{proof}
Since $A$ is simple, it is isomorphic to a matrix ring $M_{n \times n}(D)$, where $D$ is a skewfield
of finite dimension over its centre $L$, and $L$ is a finite field extension of $F$.
Let $\mathfrak{o}_{L}$ be the integral closure of $\mathfrak{o}$ in $L$.
Then $\mathfrak{o}_{L}$ is the centre of $\Lambda$ and $M$ is also an $\mathfrak{o}_{L}$-module.
If $L \otimes_{\mathfrak{o}_{L}} M = F \otimes_{\mathfrak o} M \neq 0$,
then there is no nonzero element in $\mathfrak{o}_{L}$ annihilating $M$.
This implies that $\Fit_{\Lambda}^{\max} (M) = 0$ by (i) of the Theorem.

Now suppose that $F \otimes_{\mathfrak{o}} M =0$ and choose an epimorphism $\pi: \Lambda^{k} \onto M$.
Since maximal orders are hereditary by \cite[Theorem 26.12]{MR632548},
$\ker(\pi)$ is projective by \cite[Proposition 4.3]{MR632548}.
But as $F \otimes_{\mathfrak{o}} M  =0$, we have $F \otimes_{\mathfrak{o}} \ker(\pi) \simeq A^{k}$;
thus $\ker(\pi) \simeq \Lambda^{k}$ by \cite[Theorem 18.10]{MR1972204}.
\end{proof}

We return to the proof of Theorem \ref{thm:nice-properties} (iii).
Since the reduced norm is computed component-wise, we may assume that $A$ is simple.
If $F \otimes_{\mathfrak o} M_{2} \neq 0$, then also $F \otimes_{\mathfrak o} M_{1} \neq 0$ or
$F \otimes_{\mathfrak{o}} M_{3} \neq 0$ and both sides in (\ref{eqn:additivity-in-ses}) are zero by
Lemma \ref{lem:dvr-implies-quadratic}.
If $F \otimes_{\mathfrak{o}} M_{2} = 0$, then also $F \otimes_{\mathfrak{o}} M_{1} = F \otimes_{\mathfrak{o}} M_{3} = 0$.
Hence $M_{1}$, $M_{2}$ and $M_{3}$ admit quadratic presentations by Lemma \ref{lem:dvr-implies-quadratic}.
Noting that $\mathcal{I}(\Lambda)=\zeta(\Lambda)$, the result now follows from Theorem \ref{thm:fitt-thm} (iii) and (v)
(where, as noted in \S \ref{subsec:alt-def}, $\Fitt_{\Lambda}^{\max}$ may be replaced by $\Fit_{\Lambda}^{\max}$).
\end{proof}

\begin{remark}
It is useful to be able to determine whether or not a given presentation of a finitely generated
$\Lambda$-module $M$ can be used to compute $\Fit_{\Lambda}^{\max}(M)$.
If $\Lambda$ is a direct sum of matrix rings over commutative rings, this problem is solved by
Proposition \ref{prop:matrix-over-comm-consistent}; recall that Fitting invariants over commutative rings do not depend on the chosen presentation.
If $\Lambda$ is a maximal order over a complete discrete valuation ring, we may apply Lemma \ref{lem:dvr-implies-quadratic}.
Hence we have solved this question for maximal Fitting invariants over arbitrary nice Fitting orders
over complete discrete valuation rings.
However, we note that if $\Lambda$ is isomorphic to a nice Fitting order, then it may be necessary to compute
this isomorphism explicitly, though in many cases it is possible to get away with less.
\end{remark}

\begin{example}
Let  $G$ be a finite group and let $\mathfrak{o}$ be a complete discrete valuation ring with field of fractions $F$
of characteristic zero.
Suppose the group algebra $F[G]$ decomposes into a direct sum of matrix rings over a field, i.e.,
the Schur indices of all $F$-irreducible characters of $G$ are equal to $1$.
(This happens, for example, if $G$ is dihedral or symmetric, or if $G$ is a $p$-group where $p$ is an odd prime not necessarily equal to the residue characteristic of $\mathfrak{o}$; see \cite[\S 74]{MR892316} for more on this topic.)
Let $\Lambda=\Lambda'(G,G')$ as in Definition \ref{def:max-comm-hybrid};
an explicit example is $\Lambda = \Z_{3}[A_{4}]$ as discussed in Example \ref{ex:A4}.
Now one only needs to compute the central idempotent  $e=|G'|^{-1}\Tr_{G'}$.
Indeed, $\Lambda(1-e)$ is a  direct sum of matrix rings over
complete discrete valuation rings; thus Remark \ref{rmk:fit-cent-fit-lambda} shows that
$\Fit_{\Lambda(1-e)}((1-e)M)$ is completely determined by $\Fit_{\zeta(\Lambda(1-e))}((1-e)M)$.
Since $\Lambda e$ is commutative, we therefore see that
$\Fit_{\Lambda}(M)$ is completely determined by $\Fit_{\zeta(\Lambda)}(M)$ in this case.
\end{example}

\section{Quotients by left ideals}\label{sec:quots-by-left-ideals}

We compute the maximal Fitting invariant of the quotient of a Fitting order by a left ideal in several cases.

\begin{theorem}
Let $\Lambda$ be a Fitting order and let $I$ be a left ideal of $\Lambda$. Then
\begin{enumerate}
\item We have
$\langle \nr(x) \mid x \in I \rangle_{\mathcal{I}(\Lambda)} \subset
\Fit^{\max}_{\Lambda}(\Lambda/I)$.
\item If $I$ is a principal left ideal generated by $\alpha$ then
$\Fit_{\Lambda}(\Lambda/I) \cdot \mathcal{I}(\Lambda) = \nr(\alpha) \cdot \mathcal{I}(\Lambda)$.
\item If $\Lambda$ is a direct sum of matrix rings over commutative rings, or $\Lambda$
is a nice Fitting order over a complete discrete valuation ring, then
\[
\Fit^{\max}_{\Lambda}(\Lambda/I) =
\langle \nr(x) \mid x \in I \rangle_{\zeta(\Lambda)}.
\]
\end{enumerate}
\end{theorem}

\begin{proof}
(i) Let $\{ x_{1}, \ldots, x_{r-1} \}$ be a fixed set of generators of $I$ and let $x_{r}$ be an arbitrary element of $I$.
Then there exists a presentation of the form
\[
\Lambda^{r} \stackrel{h}{\longrightarrow} \Lambda \onto \Lambda/I,
\]
where $(x_{1}, \ldots, x_{r})^{t}  \in M_{r \times 1}(\Lambda)$ is the matrix representing $h$.
Then we have $\nr(x_{r}) \in \Fit_{\Lambda}(h) \subset \Fit_{\Lambda}^{\max}(\Lambda/I)$.
Since $x_{r}$ was arbitrary, this gives the desired containment.

(ii) Let $\Lambda \stackrel{h}{\longrightarrow} \Lambda \onto \Lambda/I$ be the presentation
given by right multiplication by $\alpha$. Then since $h$ is a quadratic presentation we have
\[
\Fit_{\Lambda}^{\max}(\Lambda/I) = \Fit_{\Lambda}(h) \cdot  \mathcal{I}(\Lambda) = \nr(\alpha) \cdot \mathcal{I}(\Lambda),
\]
where the first equality follows from Theorem \ref{thm:fitt-thm} and equation \eqref{eqn:Fitt-vs-Fit}.

(iii) If $\Lambda$ is a direct sum of matrix rings over commutative rings then the result follows
from Proposition \ref{prop:matrix-ring-ideal-det}.
Thus it remains to consider the case where $\Lambda$ is a maximal order over a complete discrete
valuation ring; the result follows from Lemma \ref{lem:dvr-implies-quadratic} and part (ii) above.
\end{proof}

\section{Annihilation and change of order}\label{sec:ann-and-change-of-order}

\subsection{Conductors and annihilation}
We give annihilation results in terms of conductors.
For background material on conductors, we refer the reader to \cite[\S 27]{MR632548}.

Let $\Lambda \subset \Gamma \subset \Lambda'$ be Fitting orders in an algebra $A$,
such that $\Lambda'$ is a maximal order over the relevant Fitting domain.
Let $M$ be a finitely generated $\Lambda$-module.

\begin{definition}
We define
\begin{eqnarray*}
(\Gamma:\Lambda)_{l} &=& \{ x \in \Gamma \mid  x\Gamma \subset \Lambda\}
= \textrm{largest right $\Gamma$-module in $\Lambda$},\\
(\Gamma:\Lambda)_{r} &=& \{ x \in \Gamma \mid  \Gamma x \subset \Lambda\}
= \textrm{largest left $\Gamma$-module in $\Lambda$},
\end{eqnarray*}
and say that $(\Gamma:\Lambda)_{l}$ (resp.\ $(\Gamma:\Lambda)_{r}$) is the \emph{left} (resp.\ \emph{right})
\emph{conductor of $\Gamma$ into $\Lambda$}.
We define the \emph{central conductor} of $\Gamma$ over $\Lambda$ to be
\[
\mathcal{F}(\Gamma,\Lambda) = \{ x \in \zeta(\Gamma) \mid x\Gamma \subset \Lambda \} =
\zeta(\Gamma) \cap (\Gamma:\Lambda)_{l} = \zeta(\Gamma) \cap (\Gamma:\Lambda)_{r}.
\]
\end{definition}

\begin{prop}\label{prop:conductor-annihilation}
If $\mathcal{H}(\Gamma)=\zeta(\Gamma)$ then
$\mathcal{F}(\Gamma,\Lambda) \subset \mathcal{H}(\Lambda)$ and so
we have
\[
\mathcal{F}(\Gamma,\Lambda) \cdot \Fit_{\Lambda}^{\max}(M) \subset \Ann_{\zeta(\Lambda)}(M).
\]
\end{prop}

\begin{proof}
Let $x \in \mathcal{F}(\Gamma,\Lambda)$.
Fix $b \in \N$ and let $H \in M_{b \times b}(\Lambda)$.
Then $H \in M_{b \times b}(\Gamma)$ so $H^{\ast} \in M_{b \times b}(\Gamma)$
since $1 \in \zeta(\Gamma)=\mathcal{H}(\Gamma)$ by hypothesis.
By definition of $\mathcal{F}(\Gamma,\Lambda)$ we have $x H^{\ast} \in M_{b \times b}(\Lambda)$.
Since $b$ and $H$ were arbitrary, we have shown that $x \in \mathcal{H}(\Lambda)$.
Therefore $\mathcal{F}(\Gamma,\Lambda) \subset \mathcal{H}(\Lambda)$
and the result now follows from Theorem \ref{thm:fitt-ann}.
\end{proof}

\begin{corollary}
We have $\mathcal{F}(\Lambda', \Lambda) \cdot \Fit^{\max}_{\Lambda}(M) \subset \Ann_{\zeta(\Lambda)}(M)$.
\end{corollary}

In fact we can improve this slightly:
\begin{prop} \label{prop:conductor-of-centres-annihilates}
We have $\mathcal{F}(\zeta(\Lambda'), \zeta(\Lambda)) \subset \mathcal{H}(\Lambda)$ and so
\[
\mathcal{F}(\zeta(\Lambda'), \zeta(\Lambda)) \cdot \Fit_{\Lambda}^{\max}(M) \subset \Ann_{\zeta(\Lambda)}(M).
\]
\end{prop}

\begin{proof}
Let $n \in \N$ and let $H \in M_{n \times n}(\Lambda)$.
Then recalling the notation of \S \ref{subsec:nr} and \S \ref{subsec:adjoints},
the generalised adjoint matrix $H^{\ast}$ was defined to be
\[
H^{\ast} = \sum_{i=1}^{t} (-1)^{m_{i}-1} \sum_{j=1}^{m_{i}} \alpha_{ij} H_{i}^{j-1},
\]
where $m_{i} = n_{i} \cdot s_{i} \cdot n \in \N$, $H_{i} = H e_{i}$ and $\alpha_{ij} \in \mathfrak{o}_{i}'$,
$1 \leq i \leq t$, $1 \leq j \leq m_{i}$.
We put $m = \max_{1\leq i \leq t} (m_{i})$ and for $1 \leq i \leq t$, $1 \leq j \leq m$ we define
\[
    \tilde \alpha_{ij} = \left\{ \begin{array}{lll}
    \alpha_{ij} & \mbox{ if } & j \leq m_i\\
    0 & \mbox{ if } & j > m_i.
    \end{array} \right.
\]
Then we may write
\[
H^{\ast} = \sum_{j = 1}^{m} H^{j-1} \sum_{i=1}^{t} (-1)^{m_{i}+1} \tilde \alpha_{ij} e_{i} = \sum_{j = 1}^{m} H^{j-1} \cdot \lambda_{j}',
\]
where $\lambda_{j}' = \sum_{i=1}^{t} (-1)^{m_{i}+1} \tilde{\alpha}_{ij} e_{i}$ belongs to $\oplus_{i=1}^{t}
\mathfrak{o}_{i}' = \zeta(\Lambda')$.
Now it is clear that for any $x \in \mathcal{F}(\zeta(\Lambda'), \zeta(\Lambda))$ we have
\[
    x \cdot H^{\ast} = \sum_{j = 1}^m H^{j-1} \cdot x \cdot \lambda_j' \in M_{n \times n}(\Lambda)
\]
as desired.
\end{proof}

\begin{remark}\label{rmk:bounds-on-mathcalH}
As noted in \S \ref{subsec:adjoints} we have
$\mathcal{I}(\Lambda) \cdot \mathcal{H}(\Lambda) \subset \zeta(\Lambda)$, and so
$\mathcal{H}(\Lambda) \subset \mathcal{F}(\mathcal{I}(\Lambda), \zeta(\Lambda))$. In particular,
if $\mathcal{I}(\Lambda) = \zeta(\Lambda')$ then
$\mathcal{H}(\Lambda) = \mathcal{F}(\zeta(\Lambda'), \zeta(\Lambda))$ by Proposition \ref{prop:conductor-of-centres-annihilates}.
\end{remark}

\begin{remark}\label{rmk:arith-more-explicit}
As mentioned in the introduction, the motivation behind the theory of noncommutative Fitting invariants 
comes from arithmetic.
In \cite{derivatives-Artin-L} and \cite{organising} important arithmetic annihilation results and conjectures are given.
Let $p$ be prime and $G$ be a finite group;
then $\mathcal{A}_{p}(G)$ in \cite[\S 2.1.2]{organising} is defined to be equal to $\mathcal{H}(\Z_{p}[G])$.
Hence Proposition \ref{prop:p-not-div-comm-finite} shows that
$\mathcal{A}_{p}(G)=\zeta(\Z_{p}[G])$ in the case $p \nmid |G'|$ and
Proposition \ref{prop:conductor-annihilation} can be used to compute a subset of
$\mathcal{A}_{p}(G)$ otherwise.
Thus several of the annihilation results of \cite{organising} can be made more explicit.
Similar remarks apply to $\mathcal{A}(R[G])$ in \cite[\S 2.3]{derivatives-Artin-L} and
we expect our results to apply in many other situations.
\end{remark}

\subsection{Change of order}
Let $\Lambda \subset \Gamma$ be Fitting orders in an algebra $A$, and
let $M$ be a finitely generated $\Lambda$-module.
We compare the annihilators and maximal Fitting invariants of $M$ and $\Gamma \otimes_{\Lambda} M$
under certain conditions.

\begin{prop}\label{prop:ann-conductor}
We have
$
\mathcal{F}(\Gamma, \Lambda) \cdot \Ann_{\zeta(\Gamma)}(\Gamma \otimes_{\Lambda}M)
\subset \Ann_{\zeta(\Lambda)}(M).
$
\end{prop}

\begin{proof}
Since $M$ is finitely generated there exists $r \in \N$ and a $\Lambda$-submodule $N$ of $\Lambda^{r}$
such that $M \simeq \Lambda^{r}/N$, i.e.
$N \stackrel{\iota}\hookrightarrow \Lambda^{r} \twoheadrightarrow M$ is short exact where $\iota$ denotes the inclusion map. The functor $\Gamma \otimes_{\Lambda} - $ is right exact and so
$\Gamma \otimes_{\Lambda} M \simeq  (\Gamma \otimes_{\Lambda }\Lambda^{r}) / \im (1 \otimes \iota) \simeq \Gamma^{r} / \im (1 \otimes \iota)$. Hence we have
\begin{eqnarray*}
\Ann_{\zeta(\Lambda)}(M)
&=& \{ x \in \zeta(\Lambda) \mid x \cdot M = 0 \} = \{ x \in \zeta(\Lambda) \mid x \cdot \Lambda^{r} \subset N \},
\textrm{ and}\\
\Ann_{\zeta(\Gamma)}(\Gamma \otimes_{\Lambda} M)
&=& \{ y \in \zeta(\Gamma) \mid y \cdot (\Gamma \otimes_{\Lambda} M) = 0 \} = \{ y \in \zeta(\Gamma) \mid y \cdot \Gamma^{r} \subset \im(1 \otimes \iota) \}.
\end{eqnarray*}
Let $y \in \Ann_{\zeta(\Gamma)}(\Gamma \otimes_{\Lambda} M)$ and let $z \in \Lambda^{r}$. Then since $z \in \Gamma^{r}$ we have $yz \in \im(1 \otimes \iota)$, and so
there exists $\sum_{i=1}^{s} a_{i} \otimes b_{i} \in \Gamma \otimes_{\Lambda} N$ such that
\[
\textstyle{
yz = (1 \otimes \iota)(\sum_{i=1}^{s} a_{i} \otimes b_{i}) = \sum_{i=1}^{s} a_{i} \otimes \iota(b_{i})
\in \Gamma \otimes_{\Lambda} \Lambda^{r} \simeq \Gamma^{r}.}
\]
Now let $w \in \mathcal{F}(\Gamma, \Lambda)$.
Then since $wa_{i} \in \Lambda$ for each $i$, we have
\[
\textstyle{
wyz = w(\sum_{i=1}^{s} a_{i} \otimes \iota(b_{i})) = \sum_{i=1}^{s} wa_{i} \otimes \iota(b_{i})
= 1 \otimes (\sum_{i=1}^{s} wa_{i}\iota(b_{i})) }
\in  \Gamma \otimes_{\Lambda} \Lambda^{r} \simeq \Gamma^{r}.
\]
But $\sum_{i=1}^{s} wa_{i} \iota(b_{i}) \in N$ so under the identification
$\Gamma \otimes_{\Lambda} \Lambda^{r} \simeq \Gamma^{r}$, $a \otimes b \mapsto a \cdot b$ we have
$wyz \in N$.
Since $z \in \Lambda^{r}$ was
arbitrary, we have that $wy \in \Ann_{\zeta(\Lambda)}(M)$, as desired.
\end{proof}

\begin{corollary}\label{cor:ann-conductor}
If $\mathcal{H}(\Gamma)=\zeta(\Gamma)$ then
$
\mathcal{F}(\Gamma,\Lambda) \cdot \Fit_{\Gamma}^{\max}(\Gamma \otimes_{\Lambda} M) \subset \Ann_{\zeta(\Lambda)}(M).
$
\end{corollary}

\begin{proof}
By Theorem \ref{thm:fitt-ann} we have
$\Fit_{\Gamma}^{\max}(\Gamma \otimes_{\Lambda} M) \subset \Ann_{\zeta(\Gamma)}(\Gamma \otimes_{\Lambda} M)$.
\end{proof}

\begin{prop}\label{prop:fit-conductor}
We have
$
\Fit^{\max}_{\Lambda}(M) \cdot \mathcal{U}(\Gamma) \subset \Fit^{\max}_{\Gamma}(\Gamma \otimes_{\Lambda} M).
$
\end{prop}

\begin{proof}
Consider a presentation
$\Lambda^{a} \stackrel{h}{\longrightarrow} \Lambda^{b} \twoheadrightarrow M$.
Then by right exactness of $\Gamma \otimes_{\Lambda} - $ we have a presentation
$\Gamma^{a} \stackrel{1 \otimes h}{\longrightarrow} \Gamma^{b} \twoheadrightarrow \Gamma \otimes_{\Lambda} M$.
Note that the matrices in $M_{a \times b}(\Gamma)$ representing  $h$ and $1 \otimes h$ are in fact equal and so $S_{b}(h)=S_{b}(1 \otimes h)$. Now
\begin{eqnarray*}
\Fit_{\Lambda}(h) &=& \langle \nr(H) \mid H \in S_{b}(h) \rangle_{\mathcal{U}(\Lambda)}
\quad \textrm{ and } \\
\Fit_{\Gamma}(1 \otimes h) &=& \langle \nr(H) \mid H \in S_{b}(1 \otimes h) \rangle_{\mathcal{U}(\Gamma)}.
\end{eqnarray*}
We also have $\mathcal{U}(\Lambda) \subset \mathcal{U}(\Gamma)$.
Therefore $\Fit_{\Lambda}(h) \cdot \mathcal{U}(\Gamma) = \Fit_{\Gamma}(1 \otimes h)$
and it follows that
$\Fit_{\Lambda}^{\max}(M) \cdot \mathcal{U}(\Gamma) \subset
\Fit_{\Gamma}^{\max}(\Gamma \otimes_{\Lambda} M)$, as desired.
\end{proof}

\begin{theorem}\label{thm:fit-ann-conductor}
Suppose that $\Gamma$ is nice.
Then we have
\[
\mathcal{F}(\Gamma,\Lambda) \cdot \Fit_{\Lambda}^{\max}(M) \subset
\mathcal{F}(\Gamma,\Lambda) \cdot \Fit_{\Gamma}^{\max}(\Gamma \otimes_{\Lambda} M) \subset \Ann_{\zeta(\Lambda)}(M) .
\]
\end{theorem}

\begin{proof}
If $\Gamma$ is nice then by Proposition \ref{prop:nice-equal} we have
$\zeta(\Gamma)=\mathcal{H}(\Gamma)=\mathcal{U}(\Gamma)$, and so the desired result is now
the combination of Corollary \ref{cor:ann-conductor} and Proposition \ref{prop:fit-conductor}.
\end{proof}

\begin{remark}
Suppose that $\Gamma$ is nice.
If one wishes to compute $\zeta(\Lambda)$-annihilators of $M$ using the central conductor
$\mathcal{F}(\Gamma, \Lambda)$, then Theorem \ref{thm:fit-ann-conductor}
shows one may as well first extend scalars to $\Gamma$, allowing one to take advantage
of the useful properties of maximal Fitting invariants over nice Fitting orders
(indeed, one may also obtain more annihilators this way).
\end{remark}

\subsection{Conductors in the group ring case}\label{subsec:conductors-group-ring-case}
Let $G$ be a finite group and let $\mathfrak{o}$ be a complete discrete valuation ring
with ring of fractions $F$. Suppose that $|G|$ is invertible in $F$.
Let $\Gamma$ be a nice Fitting order containing the group ring $\Lambda :=\mathfrak{o}[G]$.
We may write
\[
\Gamma = \bigoplus_{i=1}^{k} \Gamma_{i},
\]
where $\Gamma_{i}$ is isomorphic to either a matrix ring
$M_{n_{i} \times n_{i}}(\mathfrak{o}_{i})$ over a commutative ring $\mathfrak{o}_{i}$
(not necessarily integrally closed) or
a matrix ring $M_{n_{i} \times n_{i}}(\mathfrak{o}_{D_{i}})$ over the valuation ring
$\mathfrak{o}_{D_{i}}$ of a skewfield $D_{i}$.
In the latter case, we put $\mathfrak{o}_{i} := \zeta(\Gamma_{i}) = \zeta(\mathfrak{o}_{D_{i}})$
and denote the Schur index of $D_{i}$ by $s_{i}$.
In the former case, we put $s_{i}=1$.
In both cases, $\mathfrak{o}_{i}$ is a commutative noetherian complete local ring and we may assume that it is indecomposable.
Put
$A_{i} := F \otimes_{\mathfrak{o}} \Gamma_{i}$ so that $A := F[G] =\bigoplus_{i=1}^{k} A_{i}$.
For convenience, we also put $F_{i} = F \otimes_{\mathfrak{o}} \mathfrak{o}_{i}$ so that
$\zeta(A) = \bigoplus_{i=1}^{k} F_{i}$; note that $F_{i}$ is not necessarily a field.\\

We denote the reduced trace from $A_{i}$ to $F$ by $\tr_{i}$; then we have
\[
\tr_{i} = \Tr_{F_{i}/F} \circ \tr_{A_{i} / F_{i}},
\]
where $\Tr_{F_{i}/F}$ is the ordinary trace from $F_{i}$ to $F$, and $\tr_{A_{i} / F_{i}}$ is the reduced trace from
$A_{i}$ to $F_{i}$.
For the ordinary trace $\Tr_{A/F}$ from $A$ to $F$ we thus have
\[
\Tr_{A/F}(x) = \sum_{i=1}^{k} n_{i} s_{i} \tr_{i}(x_{i})
\]
for $x = \sum_{i=1}^{k} x_{i} \in A = \bigoplus_{i=1}^{k} A_{i}$.
Abusing notation, we define the inverse different of $\Gamma_{i}$ with respect to the reduced
trace $\tr_{i}$ to be
\[
\mathfrak{D}_{i}^{-1} := \left\{x \in A_{i} \mid \tr_{i}(x \Gamma_{i}) \subset \mathfrak{o} \right\}.
\]
In the case where $\Gamma_{i}$ is a matrix ring over the valuation ring $\mathfrak{o}_{D_{i}}$ of a
skewfield $D_{i}$, this is in fact an invertible $\Gamma_{i}$-lattice,
and $\mathfrak D_{i}$ is called the different of $\Gamma_{i}$ with respect to $\tr_{i}$.
However, we note that  $\mathfrak{D}_{i}^{-1}$ is \emph{not} invertible in general.

\begin{prop} \label{prop:leftright-conductor}
We have
$
(\Gamma : \Lambda)_{l}
= (\Gamma : \Lambda)_{r} = \bigoplus_{i=1}^{k} \frac{|G|}{n_{i} s_{i}} \mathfrak{D}_{i}^{-1}.
$
\end{prop}

\begin{proof}
This is essentially the same proof as that of \cite[Theorem 27.8]{MR632548}.
\end{proof}

\begin{corollary}\label{cor:conductor-formula}
For each $i$, let
$\mathfrak{D}^{-1}(\mathfrak{o}_{i} / \mathfrak{o}) = \left\{x \in F_{i} \mid \Tr_{F_{i}/F}(x \mathfrak{o}_{i}) \subset \mathfrak{o} \right\}$,
which is the usual inverse different if $F_{i}$ is a field with ring of integers $\mathfrak{o}_{i}$.
Then we have
\[
\bigoplus_{i=1}^{k} \frac{|G|}{n_{i} s_{i}} \mathfrak{D}^{-1}(\mathfrak{o}_{i} / \mathfrak{o})
\subset \mathcal{F}(\Gamma, \Lambda).
\]
\end{corollary}

\begin{proof}
For each $i$, we have an inclusion
\begin{equation} \label{eqn:different-inclusion}
\frac{|G|}{n_{i} s_{i}} \mathfrak{D}^{-1}(\mathfrak{o}_{i} / \mathfrak{o})
\subset \frac{|G|}{n_{i} s_{i}} \mathfrak{D}_{i}^{-1} \cap \mathfrak{o}_{i}.
\end{equation}
The result now follows
since $\zeta(\Gamma) = \bigoplus_{i=1}^{k} \mathfrak{o}_{i}$ and
$\mathcal{F}(\Gamma, \Lambda) = \zeta(\Gamma) \cap (\Gamma : \Lambda)_{l}$.
\end{proof}

\begin{remark}\label{rmk:Jacobinski-formula}
If $\Gamma$ is a maximal order and $\mathfrak o$ is the ring of integers in a local field of characteristic zero, Jacobinski's central conductor formula
\cite[Theorem 3]{MR0204538} (also see \cite[Theorem 27.13]{MR632548})
implies that the inclusion (\ref{eqn:different-inclusion}) is an equality
for each $i$; thus we have also an equality in Corollary \ref{cor:conductor-formula}.
However, the argument that shows equality can not be extended to the more general situation of nice Fitting orders,
since our notion of the inverse different does not lead to invertible lattices in general.
\end{remark}

\begin{theorem}\label{thm:compare-nice-annihilators}
Let
$\mathcal{J}
= \bigoplus_{i=1}^{k} \frac{|G|}{n_{i} s_{i}} \mathfrak{D}^{-1}(\mathfrak{o}_{i} / \mathfrak{o})$.
Then for any finitely generated $\Lambda$-module $M$ we have
\[
\textstyle{
\mathcal{J} \cdot \Fit_{\Lambda}^{\max}(M) \subset
\mathcal{J} \cdot \Fit_{\Gamma}^{\max}(\Gamma \otimes_{\Lambda} M) \subset
\Ann_{\zeta(\Lambda)}(M).
}
\]
\end{theorem}

\begin{proof}
This is the combination of Theorem \ref{thm:fit-ann-conductor} and
Corollary \ref{cor:conductor-formula}.
\end{proof}

\begin{remark}
Conductors for completed group algebras are considered in \cite{conductor}.
\end{remark}

\subsection{Explicit computations and examples in the group ring case}
We now specialise to the following situation.
Let $\mathfrak{o}$ be the ring of integers in a local field $F$ of characteristic zero 
with algebraic closure $\overline{F}$ and residue field of characteristic $p>0$.
Let $G$ be a finite group and let $\Lambda'$ be a maximal order containing $\Lambda := \mathfrak{o}[G]$.
We have a natural decomposition
\begin{equation} \label{eqn:centre-of-Lambda'}
    \zeta(\Lambda') \simeq \bigoplus_{i=1}^{k} \mathfrak{o}_{i}',
\end{equation}
where $k$ is the number of irreducible $\overline{F}$-valued characters of $G$ modulo the action of $\Gal(\overline{F}/F)$
and each $\mathfrak{o}_{i}'$ corresponds to an irreducible $\overline{F}$-valued character $\chi_{i}$.
Note that the quotient field $F_{i}$ of $\mathfrak{o}_{i}'$ equals $F_{i} = F(\chi_{i}(g) \mid g \in G)$.

\begin{prop} \label{prop:conductor-cheap-trick}
Let $\Lambda'(G,G')= \mathfrak{o}[G] e \oplus \Lambda' (1-e)$ where
$G'$ is the commutator subgroup of $G$ and $e = |G'|^{-1} \Tr_{G'}$
(as in Definition \ref{def:max-comm-hybrid}). Then we have
\[
\mathcal{F}(\Lambda'(G,G'), \Lambda) = \mathfrak o[G] \cdot \Tr_{G'} \oplus \mathcal{F}(\Lambda', \Lambda)(1-e)
= \mathfrak o[G] \cdot \Tr_{G'} \oplus  \bigoplus_{\chi(1) \not=1} \frac{|G|}{\chi(1)} \mathfrak{D}^{-1}(\mathfrak{o}_{i}' / \mathfrak{o}).
\]
\end{prop}

\begin{proof}
First observe that $\mathfrak{o}[G]e$ is commutative and so $\zeta(\Lambda'(G,G'))=\mathfrak{o}[G] e \oplus \zeta(\Lambda' (1-e))$.
Moreover, $\mathcal{F}(\Lambda'(G,G'), \Lambda)$ is an ideal $\mathcal{I} \oplus \mathcal{J}$
of $\zeta(\Lambda'(G,G'))$, so we may compute $\mathcal{I}$ and $\mathcal{J}$ separately.
Since $\Lambda'(1-e)$ is maximal and (\ref{eqn:different-inclusion}) is an equality in this case
(see Remark \ref{rmk:Jacobinski-formula}), we see that $\mathcal{J}$ is of the desired form.
Now observe that
\[
\mathcal{F}(\Lambda_{G}', \mathfrak{o}[G])e
= \mathcal{I}
= ((\Lambda_{G}' : \mathfrak{o}[G])_{l}) e
= \mathfrak{o}[G]e \cap \mathfrak{o}[G]
= \mathfrak{o}[G] \cdot \Tr_{G'}.
\]
We explain the last two equalities. By definition, $\mathcal{I}$ is the largest ideal of
$\mathfrak{o}[G]e$ contained in $\mathfrak{o}[G]$, so
$\mathcal{I} \subset \mathfrak{o}[G]e \cap \mathfrak{o}[G]$.
If $xe \in \mathfrak{o}[G]$ with $x \in \mathfrak{o}[G]$, then for any $ye$ with
$y \in \mathfrak{o}[G]$ we have $(xe)(ye) = (xe)y \in \mathfrak{o}[G]$, giving the reverse inclusion.
Let $x_{1}, \ldots, x_{r}$ be a set of representatives in $G$ of
the  quotient group $G/G'$;
then $\{ x_{1}e, \ldots, x_{r}e \}$ is an $\mathfrak{o}$-basis for $\mathfrak{o}[G]e$.
Write $G'=\{h_{1}, \ldots, h_{s} \}$; then $G=\{ h_{i}x_{j}\}_{i,j}$ is an $\mathfrak{o}$-basis for
$\mathfrak{o}[G]$.
Let $x \in \mathfrak{o}[G]e$. Then we can write
$x= \lambda_{1}x_{1}e + \cdots + \lambda_{r}x_{r}e$ where each $\lambda_{k} \in \mathfrak{o}$.
Since $e=|G'|^{-1} \Tr_{G'} = |G'|^{-1}\sum_{i=1}^{s}h_{i}$, we see that
$x \in \mathfrak{o}[G]$ if and only if $|G'|$ divides each $\lambda_{k}$ if and only if
$x \in \mathfrak{o}[G] \cdot  \Tr_{G'}$.
Therefore $\mathfrak{o}[G]e \cap \mathfrak{o}[G]  = \mathfrak{o}[G] \cdot \Tr_{G'}$.
\end{proof}

\begin{example}\label{ex:D8}
Let $D_{8} = \langle a,b \mid a^{4}=b^{2}=1, bab=a^{-1} \rangle$ be the dihedral group of order $8$,
let $\Lambda=\Z_{2}[D_{8}]$, and let $\Lambda'$ be a maximal order containing $\Lambda$.
Let $\chi_{1}, \ldots, \chi_{5}$ be the $\Q_{2}$-irreducible characters of $D_{8}$, where
$\chi_{1}(1)=\cdots=\chi_{4}(1)=1$ and $\chi_{5}(1)=2$.
Let $e_{i}$ be the primitive central idempotent associated to $\chi_{i}$.
Then $\{ 8e_{1},8e_{2},8e_{3},8e_{4},4e_{5} \}$ is a $\Z_{2}$-basis of $\mathcal{F}(\Lambda',\Lambda)$
and $\{ 1+a^{2}, a+a^{3}, b+a^{2}b, ab+a^{3}b, 4e_{5}\}$ is a $\Z_{2}$-basis of
$\mathcal{F}(\Lambda'(D_{8},D_{8}'),\Lambda)$.
By using the character table of $D_{8}$ to express one basis in terms of the other and then computing the appropriate determinant, one can show that
$[\mathcal{F}(\Lambda'(D_{8},D_{8}'),\Lambda): \mathcal{F}(\Lambda',\Lambda)]_{\Z_{2}} = 2^{4}$.
Thus using $\mathcal{F}(\Lambda'(D_{8},D_{8}'),\Lambda)$ instead of $ \mathcal{F}(\Lambda',\Lambda)$
in Proposition \ref{prop:conductor-annihilation} gives an improved annihilation result.
Almost identical reasoning applies in the case $\Lambda=\Z_{2}[Q_{8}]$,
where $Q_{8}$ is the quaternion group of order $8$.
\end{example}

We now define an $\mathfrak{o}_{i}'$-ideal $\mathfrak{A}_{i}$ by
\begin{equation}\label{eq:central-cond-correct}
    \mathfrak{A}_{i} := \langle \chi_{i}(g) \mid g \in G \rangle_{\mathfrak{o}_{i}'}.
\end{equation}
Note that $\mathfrak{A}_{i} = \mathfrak{o}_{i}'$ if the degree $\chi_{i}(1)$ of the character $\chi_{i}$ is invertible in $\mathfrak{o}_{i}'$; this in particular applies to all linear characters of $G$.

\begin{prop}\label{prop:max-center-conductor}
We have an equality
\[
\mathcal{F}(\zeta(\Lambda'), \zeta(\Lambda)) = \bigoplus_{i=1}^{k} \frac{|G|}{\chi_{i}(1)} \mathfrak{A}_{i}^{-1} \mathfrak{D}^{-1}(\mathfrak{o}_{i}' / \mathfrak{o}).
\]
\end{prop}

\begin{proof}
Let $\alpha = \sum_{i=1}^{k} \alpha_{i}$ and $\beta = \sum_{i=1}^{k} \beta_{i}$ be elements in
$\bigoplus_{i=1}^{k} \mathfrak{o}_{i}'$.
Then the above isomorphism (\ref{eqn:centre-of-Lambda'}) maps $\alpha \beta$ to the group ring element
\[
\sum_{g \in G} \sum_{i=1}^{k} \sum_{\sigma \in \Gal(F_{i} / F)} \frac{\chi_{i}(1)}{|G|} \alpha^{\sigma}_{i} \beta^{\sigma}_{i}\chi_{i}^{\sigma}(g^{-1}) g \in \zeta(\Lambda').
\]
We see that $\alpha \zeta(\Lambda') \subset \zeta(\Lambda)$ if and only if
for all $1\leq i \leq k$, $\beta_{i} \in \mathfrak{o}_{i}'$ and all $g \in G$ we have
\[
\sum_{\sigma \in \Gal(F_i / F)}\frac{\chi_{i}(1)}{|G|} \alpha^{\sigma}_{i} \beta^{\sigma}_{i} \chi_{i}^{\sigma}(g^{-1}) \in \mathfrak o.
\]
The latter condition is equivalent to
$\frac{\chi_{i}(1)}{|G|} \Tr_{F_{i} /F}(\alpha_{i} \mathfrak{A}_{i}) \subset \mathfrak{o}$
for all $1 \leq i \leq k$, i.e.,
$\alpha_{i} \in \frac{|G|}{\chi_{i}(1)} \mathfrak{A}_{i}^{-1} \mathfrak{D}^{-1}(\mathfrak{o}_{i}' / \mathfrak{o})$.
\end{proof}

\begin{corollary} \label{cor:invertible-degrees}
If the degrees of all irreducible characters of $G$ are prime to $p$, then
\[
\mathcal{F}(\Lambda', \Lambda) = \mathcal{F}(\zeta(\Lambda'), \zeta(\Lambda)).
\]
\end{corollary}

\begin{proof}
As noted above, we have $\mathfrak{A}_{i} = \mathfrak{o}_{i}'$ for all $1 \leq i \leq k$ in this case.
Hence the result follows from Proposition \ref{prop:max-center-conductor}
and Jacobinski's central conductor formula (see Remark \ref{rmk:Jacobinski-formula}).
\end{proof}

\begin{corollary} \label{cor:glueing-conductors}
We have the containment
\[
\mathfrak{o}[G] \cdot \Tr_{G'} \oplus \bigoplus_{i=1 \atop \chi_{i}(1) \neq 1}^{k} \frac{|G|}{\chi_{i}(1)} \mathfrak{A}_{i}^{-1} \mathfrak{D}^{-1}(\mathfrak{o}_{i}' / \mathfrak{o}) \subset \mathcal{H}(\Lambda).
\]
\end{corollary}

\begin{proof}
This is an immediate consequence of Propositions \ref{prop:conductor-of-centres-annihilates},
\ref{prop:conductor-cheap-trick} and \ref{prop:max-center-conductor}.
\end{proof}

\begin{example} \label{ex:D2p}
Let $p$ be an odd prime and let
$D_{2p}=\langle x, y \mid x^{p}=y^{2}=1, yx=x^{-1}y \rangle$ be the
dihedral group of order $2p$.
Let $\Lambda = \Z_{p}[D_{2p}]$ and let $\Lambda'$ be a maximal $\Z_{p}$-order containing $\Lambda$.
Following \cite[Example 7.39]{MR632548}, there is a decomposition
\begin{equation}\label{eqn:D2p-decomp}
\Q_{p}[D_{2p}] \simeq \Q_{p} \oplus \Q_{p} \oplus A_{p},
\end{equation}
where $A_{p}$ is the twisted group algebra $\Q_{p}(\zeta_{p}) \oplus \Q_{p}(\zeta_{p})y$;
here, $\zeta_{p}$ denotes a primitive $p$th root of unity and multiplication in $A_{p}$
is given by $y^{2}=1$ and $\alpha y = y \tau(\alpha)$ for $\alpha \in \Q_{p}(\zeta_{p})$,
where $\tau$ denotes the unique element in $\Gal(\Q_{p}(\zeta_{p})/\Q_{p})$ of order $2$.
The surjection $\Q_{p}[D_{2p}] \onto A_{p}$ is given by $x \mapsto \zeta_{p}$ and $y \mapsto y$.
The idempotents corresponding to \eqref{eqn:D2p-decomp} are
\[
e_{1} = \frac{1}{2p} \sum_{g \in D_{2p}} g, \quad  e_{2} = \frac{1}{2p}(1-y) \cdot \sum_{i=0}^{p-1} x^{i},
\quad e_{3} = 1 - e_{1} - e_{2}.
\]
Since $A_{p}$ is not a skewfield, there must be an isomorphism $A_{p} \simeq M_{2 \times 2}(E_{p})$,
where $E_{p}=\Q_{p}(\zeta_{p}+\zeta_{p}^{-1})$ is the unique subfield of $\Q_{p} (\zeta_{p})$
such that $[ \Q_{p}(\zeta_{p}):E_{p}]=2$.
To compute the reduced norms, however, it is more convenient to work
with the irreducible matrix representation of $A_{p}$ over $\Q_{p} (\zeta_{p})$
given by
\[
\alpha \mapsto \left(\begin{array}{cc} \alpha & 0 \\0 & \tau(\alpha) \end{array} \right), \quad
        y \mapsto \left(\begin{array}{cc} 0 & 1 \\1 & 0 \end{array} \right), \quad \alpha \in \Q_{p}(\zeta_{p}).
\]
It is now easy to check that
\[
\nr(y) = e_{1} - e_{2} - e_{3}, \quad \nr(-y) = - e_{1} + e_{2} - e_{3}.
\]
Since $\nr(1) = 1$ and $2 \in \Z_{p}^{\times}$, we conclude that
$e_{i} \in \mathcal{U}(\Lambda)$ for $i=1,2,3$.

For $r \in \N$ we have
\[
e_{3}\nr(x^{r}+x^{-r})
= \det \left(\begin{array}{cc} \zeta_{p}^{r} + \zeta_{p}^{-r} & 0 \\0 & \zeta_{p}^{r} + \zeta_{p}^{-r}\end{array}\right)
= (\zeta_{p}^{r} + \zeta_{p}^{-r})^{2} = \zeta_{p}^{2r} + \zeta_{p}^{-2r} + 2.
\]
As $p$ is odd we can choose $r \in \N$ such that $2r \equiv 1 \mod p$. Since we already know that
$e_{1}, e_{2}, e_{3} \in \mathcal{U}(\Lambda)  \subset \mathcal{I}(\Lambda)$,
this shows that $e_{3}(\zeta_{p} + \zeta_{p}^{-1}) \in \mathcal{I}(\Lambda)$.
But $\mathcal{I}(\Lambda)$ is a $\Z_{p}$-order and
$\zeta(\Lambda') \simeq \Z_{p} \oplus \Z_{p} \oplus \mathfrak{o}_{E_{p}}$,
so we conclude that $\mathcal{I}(\Lambda) = \zeta(\Lambda')$.
(In fact, with more work one can show that $x^{r}+x^{-r} \in (\Z_{p}[D_{2p}])^{\times}$
and so $\mathcal{U}(\Lambda)=\mathcal{I}(\Lambda)=\zeta(\Lambda')$.)
Since all irreducible characters have degree $1$ or $2$, Remark \ref{rmk:bounds-on-mathcalH} and Corollary \ref{cor:invertible-degrees} imply that
$\mathcal{H}(\Lambda)$ is worst possible in this case, i.e.,
$\mathcal{H}(\Lambda) = \mathcal{F}(\Lambda', \Lambda)$.
\end{example}

\begin{example}
We continue with Example \ref{ex:D8}, where $G = D_{8}$ is the dihedral group of order $8$ and $\Lambda = \mathbb{Z}_2[D_{8}]$. There is only one $\mathbb{Q}_{2}$-irreducible
non-linear character of $D_{8}$ which was denoted by $\chi_{5}$. This character is of degree two, and a computation shows that $\chi_{5}(g)$ either equals $0$ or $2$ for any $g \in D_8$;
hence in the notation of \eqref{eq:central-cond-correct} we have $\mathfrak{A}_{5} = 2 \cdot \mathbb{Z}_{2}$.
If $\Lambda'$ denotes a maximal order containing $\Lambda$ then Proposition \ref{prop:max-center-conductor}
and Remark \ref{rmk:Jacobinski-formula} (respectively) imply that
\begin{eqnarray*}
 \mathcal{F}(\zeta(\Lambda'), \zeta(\Lambda)) &=& 2^{3}(\mathbb{Z}_{2} \oplus \mathbb{Z}_{2} \oplus \mathbb{Z}_{2} \oplus \mathbb{Z}_{2}) \oplus 2 \mathbb{Z}_{2}, \\
 \textrm{and} \quad  \mathcal{F}(\Lambda', \Lambda) &=& 2^{3}(\mathbb{Z}_{2} \oplus \mathbb{Z}_{2} \oplus \mathbb{Z}_{2} \oplus \mathbb{Z}_{2}) \oplus 4 \mathbb{Z}_{2}.
\end{eqnarray*}
By Corollary \ref{cor:glueing-conductors} we find that
\[
    \mathbb{Z}_{2}[D_{8}] \cdot \Tr_{D_{8}} \oplus 2 \mathbb{Z}_{2} \subset \mathcal{H}(\Lambda).
\]
Thus by the index computation in Example \ref{ex:D8} we have
\[
[\mathbb{Z}_{2}[D_{8}] \cdot \Tr_{D_{8}} \oplus 2 \mathbb{Z}_{2} :\mathcal{F}(\Lambda', \Lambda)]_{\Z_{2}} =2^{5},
\]
and so the annihilation result given therein can be further improved slightly.
More generally, if $\Lambda = \mathbb{Z}_{2}[D_{2^{a}}]$ with $a \geq 3$,
then one can show that
\[
[ \mathcal{F}(\zeta(\Lambda'), \zeta(\Lambda)) : \mathcal{F}(\Lambda', \Lambda) ]_{\Z_{2}}
= 2^{a-2}.
\]
\end{example}

\bibliography{NonCommFitt_Bib}{}
\bibliographystyle{amsalpha}

\end{document}